\documentclass[10pt,final]{siamltex}
\usepackage{amsmath}
\usepackage{lineno} 

\usepackage{bm}
\usepackage{amssymb,version}
\usepackage{cases}
\usepackage{color}
\usepackage{verbatim}
\newtheorem{rem}{Remark}[section]

\usepackage{graphicx}
\usepackage{subfigure}
\usepackage{algorithm}
\usepackage{algpseudocode}
\usepackage{siunitx}
\usepackage{empheq}

\usepackage{hyperref}
\allowdisplaybreaks

\newcommand{\normmm}[1]{{\left\vert\kern-0.25ex\left\vert\kern-0.25ex\left\vert #1
    \right\vert\kern-0.25ex\right\vert\kern-0.25ex\right\vert}}

\begin{document}
    \title{
    Several Accelerated and Stable Pseudo-Energy-Dissipative Lagrange Multiplier Methods based on Second-order Flow and Variable Splitting
\thanks{This work is supported in part by the National Natural Science Foundation of China (Grant Nos.12271302, 12131014, 12301520), Shandong Provincial Natural Science Foundation for Outstanding Youth Scholar (Grant No. ZR2024JQ030), Taishan Scholars Climbing Program of Shandong (Grant No. TSPD20240802), and the Natural Science Foundation of
Shandong Province (Grant No. ZR2023QA033).}}

\author{
	Xiaoli Li \thanks{School of Mathematics and State Key Laboratory of Cryptography and Digital Economy Security, Shandong University, Jinan, Shandong, 250100, P.R. China. Email: xiaolimath@sdu.edu.cn}.
    \and Jianghua Liu \thanks{Corresponding Author. School of Mathematics, Shandong University, Jinan, Shandong, 250100, P.R. China. Email: 202411900@mail.sdu.edu.cn}.
    \and Zhiping Mao \thanks{School of Mathematical Sciences, Eastern Institute of Technology, Ningbo 315200, Zhejiang, China. Email: zmao@eitech.edu.cn}.
    \and Jingrong Wei \thanks{Department of Mathematics, The Chinese University of Hong Kong, Shatin, New Territories, Hong Kong, China. Email: jingronw@uci.edu}.
	\and Xuan Zhao \thanks{School of Mathematics, Shandong University, Jinan, Shandong, 250100, P.R. China. Email: zhaox@sdu.edu.cn}.      
}
\maketitle

\begin{abstract}
Based on second-order inertial dynamics, we develop two accelerated Lagrange multiplier (LM)-based optimization methods: Linear LM-based second-order flow method and linear LM-based variable and operator splitting method for unconstrained non-convex optimization problems. Relaxed and adaptive variants of both methods are proposed to prevent degeneration of the Lagrange multiplier and to enable automatic adjustment of the time step size. Within the novel second-order LM-based optimization framework, the associated pseudo-energy functional is shown to dissipate monotonically along the iterative trajectory, and convergence of the generated iterates to stationary points is rigorously established. Extensive numerical experiments on function optimization and partial differential equation benchmark problems, including applications as optimizers for physics-informed neural networks (PINNs) and deep operator networks (DeepONets), demonstrate that the proposed methods effectively escape local minima, achieve high accuracy, and exhibit strong robustness with respect to the choice of the initial learning rate during neural network training.


\end{abstract}


\section{Introduction}
Optimization algorithms play a fundamental role in a wide range of scientific and engineering applications, including function approximation, image recognition, numerical solution of partial differential equations (PDEs) and machine learning. Most of these applications are related to the optimization of neural network (NN). NN-based approaches construct a global loss function to determine suitable networks parameters. Owing to the intrinsic complexity of the neural network, the resulting loss functions are typically highly non-convex, making the optimization process particularly challenging. 



To address such non-convex optimization problems, numerous optimization algorithms have been developed, including the gradient descent (GD) method \cite{cauchy1847methode}, the stochastic gradient descent (SGD) method \cite{robbins1951stochastic}, the Nesterov accelerated gradient (NAG) method \cite{JMLR:v17:15-084}, the adaptive gradients (Adagrad) method \cite{duchi2011adaptive}, the root mean square propagation (RMSprop) method \cite{hinton2012neural}, and the adaptive delta (AdaDelta) method \cite{article}. Among these approaches, the adaptive moment estimation (ADAM) method \cite{kingma2014adam} has become one of the most widely used adaptive optimization algorithms in deep learning practice.
However, these optimization algorithms share a common limitation: their stability and computational efficiency remain highly sensitive to the choice of the fixed or initial learning rate. In practical applications, considerable manual tuning is often required to achieve satisfactory convergence behavior. This leads to a fundamental trade-off: excessively large (initial) learning rates may cause training instability or even divergence, whereas overly small (initial) learning rates generally result in prohibitively high computational costs due to slow convergence. Consequently, the development of optimization algorithms with improved robustness and computational efficiency remains a central issue in deep learning methods.

Recently, Liu et al. \cite{shen2023} proposed a novel framework that reformulates optimization problems as the search for steady-state solutions of gradient flows from a continuous perspective. Based on this framework, they further developed efficient modified energy-stable algorithms using the scalar auxiliary variable (SAV) approach. Building upon this idea, Ma et al. \cite{ma2024efficient} introduced a hybrid discretization strategy that combines the smoothed particle method (SPM) for spatial discretization with SAV-based time-stepping schemes incorporating the adaptive mechanism of the ADAM method for temporal discretization. Numerical experiments demonstrated that this strategy achieves excellent performance in applications such as function approximation, image processing, and PDEs solving. However, our numerical investigations indicate that, without the strategy combined with the proposed PM, the convergence rate becomes slow and the solution accuracy deteriorates when solving PDEs. To address the difficulties encountered by SAV-based optimization algorithms in PDEs solving, Liu et al. \cite{LIU2026114750} further proposed two Lagrange multiplier (LM)-based subspace schemes for solving PDEs. Their approach first employs an LM-based optimization algorithm to construct suitable subspaces, followed by the least-squares method to determine the weights of the final subspace layer. The further SAV-based energy-stable algorithm also includes the vector auxiliary variables (VAV) \cite{doi:10.1137/23M1611087}. Although these methods achieve high accuracy in PDEs solving and exhibit robustness with respect to the (initial) learning rate, they are constructed on the first-order gradient flows, and the resulting schemes satisfy either the modified or original energy dissipation laws. From the perspective of continuous dynamics, the first-order flows may suffer from relatively slow convergence. From the viewpoint of discrete energy evolution, enforcing modified or original energy dissipation may increase the risk of the optimization process becoming trapped in local minima. Fortunately, the second-order flow model proposed in \cite{CHEN2023111872} may be able to resolve the above issues. The authors developed a series of numerical algorithms based on the proposed second-order hyperbolic dissipative PDEs systems. Although the computational complexity of these methods is comparable to that of the gradient-flow-type methods, the corresponding numerical results exhibit significantly improved performance and robustness.

Consider the unconstrained minimization problem:
\begin{equation}\label{eq1_1}
	\min_{x \in \mathbb{R}^N} f(x),
\end{equation}
where $\mathbb{R}^N$ denotes the $N$-dimensional Euclidean space equipped with the standard $l^{2}$ inner product $(\cdot , \cdot )$ (resp., norm $\|\cdot\|$), and $f:\mathbb{R}^{N}\to\mathbb{R}$ is differentiable. We assume that a solution $x^* \in \arg \min f(x)$ exists and denote $f^* = f(x^*) > - \infty$.

For the unconstrained minimization problem \eqref{eq1_1}, a widely studied continuous-time model for accelerated optimization is the second-order inertial gradient system:
\begin{equation}\label{eq1_2}
	\ddot{x} + \eta(t)\dot{x} = -\nabla f(x),
\end{equation}
where $\eta(t)>0$ denotes the damping coefficient, $\nabla f$ denotes the gradient of $f$, and $\dot{x} = \dot{x}(t), \ddot{x} = \ddot{x}(t)$ denotes the first- and second-order derivative with respect to $t$, respectively. This class of second-order dissipative dynamical systems has been extensively studied in optimization; see, e.g., \cite{attouch2023fast,attouch2018fast,article3,JMLR:v17:15-084,chen2025second} and the references therein.
When $\eta(t) \equiv \eta$ is a fixed constant, \eqref{eq1_2} corresponds to the Heavy Ball with friction system introduced by Polyak \cite{POLYAK19641} and a convergence rate of $O(\frac{1}{t})$ is shown in \cite{7330562}. When $\eta(t)=\frac{3}{t}$, \eqref{eq1_2} corresponds to the continuous ODE model associated with NAG method and the convergence rate was improved to $O(\frac{1}{t^2})$  \cite{JMLR:v17:15-084}.

When solving optimization problems, the first-order and second-order flow models exhibit distinct energy dissipation behaviours. Specifically, the first-order gradient flows are characterized by the monotonic dissipation of the energy function $f(x)$, whereas the second-order flow \eqref{eq1_2} dissipates the following pseudo-energy through Lyapunov analysis \cite{attouch2018fast,CHEN2023111872,chen2025second}:
\begin{equation*}
	\mathcal{E}(x) = f(x) + \frac{1}{2}\left\|\dot{x}\right\|^2.
\end{equation*}
Since non-convex optimization problems typically have multiple local minima and an extremely complex solution landscape, the monotonic decrease of the energy function $f$ may hinder optimization algorithms from escaping undesirable local minima. By contrast, the inertial term contained in the pseudo-energy function introduces additional kinetic information, which facilitates escaping from local minima and potentially improves optimization performance.



Motivated by the aforementioned advantages of the second-order flow, we apply the corresponding second-order flow model to non-convex optimization problems arising in neural networks with complex structures. Based on this framework, we develop several accelerated, efficient, and stable algorithms for solving PDEs. The main contributions are as follows:
\begin{itemize}
    \item \textbf{Linear LM-based second-order flow (LLM-SOF) method.}
    We propose a novel linear Lagrange multiplier (LM)-based second-order flow
    method that preserves a pseudo-energy dissipation law. The proposed method enjoys several desirable properties: it is
    linear (requiring only one linear system to solve at each iteration), fully
    decoupled, unconditionally energy stable, and maintains strict positivity
    of the Lagrange multiplier throughout the iteration process. We rigorously prove the convergence to stationary points for the generated iterative sequence.
    \item \textbf{Linear LM-based variable and operator splitting (LLM-VOS) method with accelerated convergence.}
Motivated by the work of Chen et al. \cite{chen2025acceleratedgradientmethodsvariable}, we reformulate the second-order flow into a variable and operator splitting (VOS) flow and develop the corresponding LM-based VOS method, which preserves the pseudo-energy dissipation law while significantly accelerating convergence. The VOS formulation shares the same desirable properties, linear, fully decoupled, unconditionally energy stable, and strict positivity of the Lagrange multiplier, and enables more efficient computation and improved convergence compared to the standard second-order flow.
    \item \textbf{Relaxed and adaptive variants with robustness.}
    To handle the complexity of non-convex optimization problems and enhance convergence, we develop relaxed and adaptive variants of both the LLM-SOF and LLM-VOS methods while preserving the pseudo-energy dissipation property. The relaxed version prevents the Lagrange multiplier from becoming too small, ensuring numerical stability and avoiding slow convergence, while the adaptive version
    allows for automatic adjustment of the time step size.
\end{itemize}

We apply the proposed approaches to several representative non-convex optimization problems, including the Rosenbrock function, the Rastrigin function, PINNs for solving the Poisson equation, and DeepONets for solving the Burgers' equation. Numerical experiments demonstrate that the proposed methods achieve relatively fast convergence while maintaining strong robustness and high accuracy, even under large initial learning rates. In contrast, those first-order flow-based methods as well as the ADAM method either converge slowly or become unstable, leading to training failure or significant numerical errors.

The remainder of this paper is organized as follows. In Section 2, we develop two LM-based methods inspired by accelerated gradient flows and introduce their relaxed and adaptive variants. In Section 3, we establish convergence results for the proposed algorithms. In Section 4, we present numerical experiments to demonstrate their accuracy, efficiency, and robustness. Concluding remarks are given in Section 5.

\section{Linear LM-based second-order flow and methods}\label{section2}
In this section, we propose several efficient, robust and pseudo-energy stable LM-based schemes, and develop relaxed and adaptive variants.

\subsection{Linear LM-based second-order flow (LLM-SOF) algorithm}\label{section2_1}

Introducing the velocity variable $v=\dot{x}$, the second-order inertial gradient flow \eqref{eq1_2} can be written as the first-order ODE system:
\begin{equation}\label{eq2_1}
	\begin{cases} 
    \dot{v} = -\nabla f(x) - \eta(t)v, \\
    \dot{x} = v.
	\end{cases}
\end{equation}

Define the pseudo-energy
\begin{equation*}\label{eq2_3}
	\mathcal{E}(x, v) := f(x) + \frac{1}{2}\|v\|^2.
\end{equation*}
Throughout the paper, we assume that $\mathcal E > 0$ without loss of generality since adding a sufficiently large constant $C > -f^*$ to $f$ does not change the optimization problem or its gradient. The minimizers of $\mathcal{E}(x, v)$ are related to the minimizers of $f(x)$:
$$
 \mathcal{E}(x^*, v^*) \leq \mathcal{E}(x, v), \quad \forall~ x, v \in \mathbb{R}^n, 
$$
with $x^* \in \arg \min f(x)$ and $v^* = 0$.

The ODE system \eqref{eq2_1} is dissipative with respect to the pseudo-energy in the sense that
\begin{align}\label{eq2_2}
	\frac{d\mathcal{E}}{dt} = -\eta(t) \|v\|^2\leq0.
\end{align}

We propose the linear LM-based second-order inertial gradient flow (LLM-SOF):
\begin{subequations}\label{eq2_4}
\begin{empheq}[left=\empheqlbrace]{align}
   \dot{v}  &= - \xi(t) \nabla f(x) - S(\mathcal{L}x - \mathcal{L}x) -\eta(t)v, \\
    \dot{x}  &= v, \\
		\frac{d\mathcal{E}}{dt} &= -\eta(t)\|v\|^2, \label{eq_LMdampedflow_c}
\end{empheq}
\end{subequations}
where $\xi(t)$ denotes the global Lagrange multiplier, $\mathcal{L}$ is a linear semi positive-definite operator and $S$ is a positive constant. The term $S(\mathcal Lx-\mathcal Lx)$ vanishes at the continuous level, but its two copies will be treated at different time levels in the discrete algorithm. Noted that \eqref{eq_LMdampedflow_c} implies 
$$
(1 -\xi(t)) (v(t), \nabla f(x(t))) = 0, \quad \forall~ t > 0. 
$$
The LLM-SOF \eqref{eq2_4} generalized the second-order inertial gradient flow by imposing $\xi(t) = 1$ whenever $(v(t), \nabla f(x(t))) \neq 0$ and preserved the energy decay. If $\xi(t) \equiv 1$, LLM-SOF \eqref{eq2_4} reduced to the second-order inertial gradient flow \eqref{eq2_1}. 

Discretizing \eqref{eq2_4} with a fixed time step $\Delta t>0$, we propose the following linear LM-based second-order-flow (LLM-SOF) scheme:
\begin{subequations}\label{eq2_5}
\begin{empheq}[left=\empheqlbrace]{align}
&\frac{v^{k+1} - \frac{\xi^{k+1}}{\xi^{k}}v^k}{\Delta t} = - \xi^{k+1}\nabla f(x^k) - S(\mathcal{L}(x^{k+1} - x^k)) -\eta^{k+1}v^{k+1}, \label{eq2_5a}\\
&\frac{x^{k+1} - x^k}{\Delta t} = v^{k+1},\label{eq2_5b}\\
&\frac{\xi^{k+1}\mathcal{E}^k - \xi^k\mathcal{E}^{k-1}}{\Delta t} = -\eta^{k+1}\left\|v^{k+1}\right\|^2,\label{eq2_5c}
\end{empheq}
\end{subequations}
with $\mathcal E^k:= \mathcal E(x^k, v^k)$ and initial values $\xi^0=1,x^{-1}=x^0$. The scheme \eqref{eq2_5} is an explicit-implicit discretizaton of the LLM-SOF \eqref{eq2_4}. To facilitate the stability and energy decay of the algorithms, we incorporate multiplier terms in the discrete algorithm. In \eqref{eq2_5a}, $v^k$ is multiplied by the ratio of the multiplier $\frac{\xi^{k+1}}{\xi^{k}}$. In \eqref{eq2_5c}, the energy decay in the discrete level is modified to $\tilde{\mathcal{E}}(x^k, v^k, \xi^{k+1}) = \xi^{k+1}\mathcal{E}(x^k, v^k)$.

\begin{rem}
    By taking $\Delta t \to 0$, the corresponding ODE model of the LLM-SOF scheme \eqref{eq2_5} can be formulated as 
    \begin{equation}\label{eqLMODE-2}
    \left \{ \begin{aligned}
     &    \dot{v} - \dot{\xi} \frac{v}{\xi}  =  -\xi \nabla f - S(\mathcal{L}x- \mathcal{L}x)  -\eta(t) v, \\
   & \dot{x} = v,\\
   & \frac{d (\xi \mathcal{E})}{d t} = -\eta(t) \left\|v\right\|^2,
    \end{aligned} \right.
    \end{equation}
provided $\xi(t) \neq 0, \forall~ t \geq 0$. With algebraic computation, \eqref{eqLMODE-2} can be understood as an second-order inertial gradient flow with a time-dependent Lagrange multiplier 
$$
\dot{\xi} = - \frac{1}{ \mathcal E + \|v\|^2} (1-\xi)[\eta \|v\|^2 +\xi(\nabla f(x), v)].
$$
A trivial dynamic for $\xi(t)$ is choosing $\xi \equiv 1$, i.e.,  $\dot{\xi} \equiv 0$,  then \eqref{eqLMODE-2} reduces to the second-order inertial gradient flow \eqref{eq2_1}. As our main focus is to design discrete algorithms, we leave the study of the ODE system to interested readers.
\end{rem}

The implementation of the LLM-SOF scheme \eqref{eq2_5} requires one solve of a linear system involving $\mathcal L$ and algebraic manipulations. Using \eqref{eq2_5a} and \eqref{eq2_5b}, we can set
\begin{equation*}\label{eq2_6}
	x^{k+1}=x^k+\Delta t\xi^{k+1}z^{k+1},
\end{equation*}
where the auxiliary variable $z^{k+1} := v^{k+1}/\xi^{k+1}$ is given by the following
\begin{equation}\label{eq2_7}
	\left (\frac{I}{\Delta t}+S\Delta t\mathcal{L}+\eta^{k+1} I \right )z^{k+1}=\frac{v^k}{\Delta t \xi^{k}}-\nabla f(x^k).
\end{equation}
We obtain $z^{k+1}$ by solving the linear system \eqref{eq2_7} with constant coefficients. Then substituting $v^{k+1}=\xi^{k+1}z^{k+1}$ into \eqref{eq2_5c}, we get the following algebraic equation:
\begin{equation}\label{eq2_8}
\eta^{k+1}\left\|z^{k+1}\right\|^2(\xi^{k+1})^2+\frac{\mathcal{E}^k}{\Delta t}\xi^{k+1}-\frac{\xi^k\mathcal{E}^{k-1}}{\Delta t}=0.
\end{equation}
If $\left\|z^{k+1}\right\| \neq 0$, \eqref{eq2_8} is a quadratic equation for $\xi^{k+1}$. Using the root formula,
\begin{equation*}\label{eq2_9}
	\xi^{k+1}=\frac{-\frac{\mathcal{E}^k}{\Delta t} \pm \sqrt{(\frac{\mathcal{E}^k}{\Delta t})^2+4\eta^{k+1}\left\|z^{k+1}\right\|^2\frac{\xi^k\mathcal{E}^{k-1}}{\Delta t}}}{2\eta^{k+1}\left\|z^{k+1}\right\|^2}.
\end{equation*}
Since we expect $\xi^{k+1}$ to approximate $1$, we always choose the positive root, i.e., 
\begin{equation}\label{eq2_10}
	\xi^{k+1}=\frac{-\frac{\mathcal{E}^k}{\Delta t} + \sqrt{(\frac{\mathcal{E}^k}{\Delta t})^2+4\eta^{k+1}\left\|z^{k+1}\right\|^2\frac{\xi^k\mathcal{E}^{k-1}}{\Delta t}}}{2\eta^{k+1}\left\|z^{k+1}\right\|^2}>0.
\end{equation}
If $\left\|z^{k+1}\right\| = 0$, \eqref{eq2_8} degenerates to a linear equation and we get $\xi^{k+1} = \xi^k \mathcal E^{k-1}/\mathcal E^{k} $. To summarize the process above, we give the LLM-SOF scheme in Algorithm \ref{alg:SOF_scheme}.
\begin{algorithm}
	\caption{The Linear LM-based Second-order Flow (LLM-SOF) Scheme}
	\label{alg:SOF_scheme}
	\begin{algorithmic}[1]
		\State \textbf{Inputs:} $x^0, v^0 \in \mathbb{R}^n, \Delta t > 0, \xi^0 =1, \{\eta^k\}_k > 0$
        \State $x^{-1} = x^0, v^{-1} = v^0$
		\For{$k = 0, 1, \dots, N - 1$}
		\State $z^{k+1} = (\frac{I}{\Delta t}+S\Delta t\mathcal{L}+\eta^{k+1}I)^{-1}(\frac{v^k}{\Delta t\xi^k}-\nabla f(x^k)) $
		\State $\displaystyle \xi^{k+1}=\begin{cases}
		     \frac{-\frac{\mathcal{E}^k}{\Delta t} + \sqrt{(\frac{\mathcal{E}^k}{\Delta t})^2+4\eta^{k+1}\left\|z^{k+1}\right\|^2\frac{\xi^k\mathcal{E}^{k-1}}{\Delta t}}}{2\eta^{k+1}\left\|z^{k+1}\right\|^2}, \quad \|z^{k+1}\| \neq 0\\
             \frac{ \xi^k \mathcal E^{k-1}}{\mathcal E^{k}}, \text{ otherwise }
		\end{cases}
        $
		\State $x^{k+1}=x^k+\Delta t\xi^{k+1}z^{k+1}$
		\EndFor
		\State \Return $x^N$
	\end{algorithmic}
\end{algorithm}

The positivity of the multiplier and the energy stable property follow directly from \eqref{eq2_10}.

\begin{theorem}\label{th2_1}
	The Lagrange multiplier $\xi^{k}$ in Algorithm \ref{alg:SOF_scheme} is positive for $k \geq 0$. Furthermore, Algorithm \ref{alg:SOF_scheme} is unconditionally energy stable in the sense that the modified pseudo-energy
	\begin{equation}\label{eq2_11}
		\tilde{\mathcal{E}}^{k+1} - \tilde{\mathcal{E}}^k = -\Delta t \eta^{k+1} \left\| v^{k+1} \right\|^2 \leq 0, \quad \forall~ k \geq 0
	\end{equation}
    where \(\tilde{\mathcal{E}}^{k+1} := \xi^{k+1} \mathcal{E}(x^k, v^k)\) is nonnegative for all $k \geq 0$.
\end{theorem}


However, $\xi^{k}$ in the LLM-SOF scheme may go to 0 if $\Delta t$ is too large, which might affect accuracy and lead to a local minimum. To resolve this issue, inspired by the relaxation method in \cite{li2025class}, we compute the relaxed LM by
\begin{align}\label{eq15}
	\xi^{k+1}=\min \left \{ \frac{ \xi^{k}\mathcal{E}^{k-1} } {\mathcal{E}^{k}}, \frac{\mathcal{E}^{k+1}} {\mathcal{E}^{k} } \right\}.
\end{align}
With relaxation, we propose the relaxed LLM-SOF scheme (RSOF) in Algorithm \ref{alg:RSOF_scheme}. It is straightforward to verify Algorithm \ref{alg:RSOF_scheme} is still unconditionally energy stable.

\begin{algorithm}
	\caption{The Relaxed SOF (RSOF) Scheme}
	\label{alg:RSOF_scheme}
	\begin{algorithmic}[1]
		\State \textbf{Inputs:} $x^0, v^0 \in \mathbb{R}^n, \Delta t >0, \xi^0 =1$
        \State $x^{-1} = x^0, v^{-1} = v^0$
		\For{$k = 0, 1, \dots, N - 1$}
		\State $z^{k+1} = (\frac{I}{\Delta t}+S\Delta t\mathcal{L}+\eta^{k+1} I)^{-1}(\frac{v^k}{\Delta t\xi^k}-\nabla f(x^k)) $
		\State $\widetilde{\xi}^{k+1}=\begin{cases}
		     \frac{-\frac{\mathcal{E}^k}{\Delta t} + \sqrt{(\frac{\mathcal{E}^k}{\Delta t})^2+4\eta^{k+1}\left\|z^{k+1}\right\|^2\frac{\xi^k\mathcal{E}^{k-1}}{\Delta t}}}{2\eta^{k+1}\left\|z^{k+1}\right\|^2}, \quad \|z^{k+1}\| \neq 0\\
             \frac{ \xi^k \mathcal E^{k-1}}{\mathcal E^{k}}, \text{ otherwise }
		\end{cases}$
		\State $x^{k+1}=x^k+\Delta t\widetilde{\xi}^{k+1}z^{k+1}$
		\State $\xi^{k+1}=\min \left \{ \frac{ \xi^{k}\mathcal{E}^{k-1} } {\mathcal{E}^{k}}, \frac{\mathcal{E}^{k+1}} {\mathcal{E}^{k} } \right\}$
		\EndFor
		\State \Return $x^N$
	\end{algorithmic}
\end{algorithm}

In order to improve the efficiency and accuracy when solving complicated problems, we further propose an adaptive time step variant of the LLM-SOF scheme based on \cite{shen2023}. The main idea is that, for solving the system \eqref{eq2_5}, $\xi^k$ should be as close to 1 as possible for the sake of the time accuracy. But for a minimization problem, there is no time accuracy issue, and thus we can allow  to deviate from 1 to achieve faster convergence. However,  $\xi^k$ needs to be away from zero to avoid slow convergence. Therefore, when $\xi^k <\gamma$ and $\Delta t^k  > \Delta t_{min}$, we set $\Delta t^{k+1} = \max\{\xi^k\Delta t^k,\Delta t_{min}\}$. Otherwise, we enlarge $\Delta t^{k+1}$ by multiplying $\Delta t^{k}$ by the factor $\rho$.  The adaptive time step strategy is given in Algorithm \ref{alg:ARSOF_scheme} and we call it the adaptive relaxed LLM-SOF (ARSOF) method. $\Delta t_{min}$ is the lower bound of step-size with the default value $0.01$, $\rho > 1$ is an adaptive constant to enlarge step size, and $\gamma > 0$ is the threshold for the Lagrange multiplier $\xi$.


The relaxed and adaptive variants are both unconditionally energy stable as stated in the following theorem. Numerically, our proposed algorithms converge faster with adaptive time step size.
\begin{theorem}\label{thm:RSOFenergy_decay}
	The Lagrange multiplier $\xi^{k}$ in Algorithm \ref{alg:RSOF_scheme} and Algorithm \ref{alg:ARSOF_scheme} are positive for $k \geq 0$. Furthermore, Algorithm \ref{alg:RSOF_scheme} and Algorithm \ref{alg:ARSOF_scheme}  are unconditionally energy stable in the sense that the modified pseudo-energy is nonincreasing:
	\begin{equation*}
		\tilde{\mathcal{E}}^{k+1} \leq \tilde{\mathcal{E}}^k \leq \cdots \leq   \tilde{\mathcal{E}}^0, \quad  k \geq 0,
	\end{equation*}
    where \(\tilde{\mathcal{E}}^{k+1} := \xi^{k+1} \mathcal{E}(x^k, v^k)\) is nonnegative for all $k \geq 0$.
\end{theorem}


\begin{algorithm}
	\caption{The Adaptive Relaxed SOF (ARSOF) Scheme}
	\label{alg:ARSOF_scheme}
	\begin{algorithmic}[1]
	\State \textbf{Inputs:} $x^0, v^0 \in \mathbb{R}^n, \Delta t^0> 0, \xi^0 = 1, \rho,\gamma,\Delta t_{min} > 0$
    \State $x^{-1} = x^0, v^{-1} = v^0$
	\For{$k = 0, 1, \dots, N - 1$}
	\If{$\xi^k<\gamma \ and\  \Delta t^k>\Delta t_{min}$}
		\State $\Delta t^{k+1}=\max\{\xi^k\Delta t^k,\Delta t_{min}\}$
	\Else
		\State $\Delta t^{k+1}=\rho\Delta t^{k}$
	\EndIf
	\State $\Delta t=\Delta t^{k+1}$
	\State $z^{k+1} = (\frac{I}{\Delta t}+S\Delta t\mathcal{L}+\eta^{k+1}I)^{-1}(\frac{v^k}{\Delta t\xi^k}-\nabla f(x^k)) $
	\State $\widetilde{\xi}^{k+1}=\begin{cases}
		     \frac{-\frac{\mathcal{E}^k}{\Delta t} + \sqrt{(\frac{\mathcal{E}^k}{\Delta t})^2+4\eta^{k+1}\left\|z^{k+1}\right\|^2\frac{\xi^k\mathcal{E}^{k-1}}{\Delta t}}}{2\eta^{k+1}\left\|z^{k+1}\right\|^2}, \quad \|z^{k+1}\| \neq 0\\
             \frac{ \xi^k \mathcal E^{k-1}}{\mathcal E^{k}}, \text{ otherwise }
		\end{cases}$
	\State $x^{k+1}=x^k+\Delta t\widetilde{\xi}^{k+1}z^{k+1}$
	\State $\xi^{k+1}=\min \left \{ \frac{ \xi^{k}\mathcal{E}^{k-1} } {\mathcal{E}^{k}}, \frac{\mathcal{E}^{k+1}} {\mathcal{E}^{k} } \right\}$
	\EndFor
	\State \Return $x^N$
\end{algorithmic}
\end{algorithm}


\subsection{Linear LM-based variable and operator splitting (LLM-VOS) algorithm}\label{section2_2}
Introducing an auxiliary variable $y$, which converges to the primal variable $x$ but has its own dynamics, the  inertial gradient flow can be transformed \eqref{eq2_1} into the following system:
\begin{equation*}\label{eq2_20}
	\begin{cases}
		\dot x=y-x,\\
		\dot y=(1-\eta(t))(y-x)-\nabla f(x),
	\end{cases}
\end{equation*}
where $v=y-x$. The splitting of variables to $x$ and $y$ was inspired by the variable and operator splitting (VOS) flow \cite{chen2025acceleratedgradientmethodsvariable}, whose discretization yields accelerated methods for monotone inclusion problems. We introduce the Lagrange multiplier $\xi(t)$, and propose the linear LM-based VOS flow: 
\begin{equation}\label{eq2_21}
\left\{
	\begin{aligned}
	    &\dot x=y-x,\\
		&\dot y=(1-\eta(t))(y-x)-\xi(t)\nabla f(x),\\
		&\frac{d\mathcal{E}}{dt} = -\eta(t)\|y-x\|^2,
	\end{aligned}
\right.
\end{equation}
where
$$
\mathcal E(x,y):= f(x) + \frac{1}{2}\|y-x\|^2.
$$

An implicit–explicit (IMEX) discretization of \eqref{eq2_21} combined with accelerated over-relaxation (AOR) \cite{fef90db1-7db7-3d75-a07c-59056d8c440a} yields the following LLM-VOS scheme:
\begin{subequations}\label{eq2_22}
\begin{empheq}[left=\empheqlbrace]{align}
&\frac{x^{k+1}-x^k-(1-\xi^{k+1})(y^k-x^k)}{\Delta t}=2y^{k+1}-y^k-x^{k+1},\label{eq2_22a}\\
		&\frac{y^{k+1}-y^k}{\Delta t}=(1-\eta^{k+1})(y^{k+1}-x^k)-\xi^{k+1}\nabla f(x^k),\label{eq2_22b}\\
		&\frac{\xi^{k+1}\mathcal{E}^k - \xi^k\mathcal{E}^{k-1}}{\Delta t} = -\eta^{k+1}\left\|y^{k+1}-x^{k+1}\right\|^2.\label{eq2_22d}
\end{empheq}
\end{subequations}
with $\mathcal E^k:= \mathcal E(x^k, y^k)$ and initial values $\xi^0=1,x^{-1}=x^0$ and $y^{-1}=y^0$. \eqref{eq2_22a} used the extrapolation $y \approx 2y^{k+1}-y^k$, which is also known as the accelerated over-relaxation (AOR) originally in linear iterations \cite{fef90db1-7db7-3d75-a07c-59056d8c440a}. The AOR discretization is crucial for the analysis of accelerated convergence. In the primal–dual method of Chambolle and Pock \cite{chambolle2011first}, the parameter $\theta=1$ yields the same extrapolation, and He and Yuan \cite{article2} later reinterpreted the scheme as a preconditioned proximal point method, proving its convergence by establishing a contractive quadratic Lyapunov functional.

The implementation of the LLM-VOS scheme\eqref{eq2_22} only requires one solve of a linear system with scalar coefficient and algebraic manipulations, so its computational efficiency is improved compared to LLM-SOF method. Introducing the auxiliary variable $v^{k+1}:= y^{k+1} - x^{k+1}$ , the update of $x$ in \eqref{eq2_22a} is equivalent to 
\begin{equation}\label{eq2_24}
    x^{k+1} = x^k  +(1-\xi^{k+1} )(y^k - x^k) + \Delta t(v^{k+1} +y^{k+1}- y^k).
\end{equation}
The update of $y$ in \eqref{eq2_22b} is equivalent to 
\begin{equation}\label{eq2_25}
    y^{k+1} = y^k + \Delta t((1-\eta^{k+1})v^{k+1}-\xi^{k+1}\nabla f(x^k)).
\end{equation}

Subtracting \eqref{eq2_24} from \eqref{eq2_25} we get 
\begin{equation*}\label{eq2_26}
    \frac{v^{k+1} - \xi^{k+1} v^k}{\Delta t} = (1-\Delta t)( -\eta^{k+1}v^{k+1}-\xi^{k+1}\nabla f(x^k) ) - \Delta t v^{k+1}.
\end{equation*}
Divided both side by $\xi^{k+1}$ and define $z^{k+1}: = v^{k+1}/\xi^{k+1}$, we have
\begin{equation}\label{eq2_27}
    (1/\Delta t + \Delta t  + (1-\Delta t)\eta^{k+1}) z^{k+1}  = \frac{v^{k}}{\Delta t} -   (1-\Delta t)\nabla f(x^k),
\end{equation}
which implies 
\begin{equation*}\label{eq2_28}
     z^{k+1}  =\frac{1}{1/\Delta t + \Delta t + (1-\Delta t)\eta^{k+1}} \left(\frac{v^{k}}{\Delta t} -   (1-\Delta t)\nabla f(x^k)\right),
\end{equation*}
given $1/\Delta t + \Delta t  + (1-\Delta t)\eta^{k+1} \neq 0$.

Noted that $y^{k+1} - x^{k+1} =  z^{k+1} \xi^{k+1}$ and plugging in \eqref{eq2_22d}, we get 
\begin{equation}\label{eq2_29}
     \eta^{k+1} \|z^{k+1}\|^2 (\xi^{k+1})^2  + \frac{\mathcal{E}^k}{\Delta t}\xi^{k+1} - \xi^k\frac{\mathcal{E}^{k-1}}{\Delta t} = 0.
\end{equation}

The algebraic equation  of $\xi^{k+1}$ in \eqref{eq2_29} is identical to that of LLM-SOF method in \eqref{eq2_8} . Given that $\xi^{k+1}$ is computed, $y^{k+1}$ is updated first, followed by $x^{k+1}$. We summarize the LLM-VOS scheme in Algorithm \ref{alg:VOS_scheme}.		
\begin{algorithm}
	\caption{ The Linear LM-based Variable and Operator Splitting (LLM-VOS) Scheme}
	\label{alg:VOS_scheme}
	\begin{algorithmic}[1]
    
		\State \textbf{Inputs:} $x^0, y^0 \in \mathbb{R}^n, \Delta t > 0, \xi^0 =1$
        \State $x^{-1} = x^0, y^{-1} = y^0, v^0 = x^0 - y^0$
		\For{$k = 0, 1, \dots, N - 1$}
		\State $z^{k+1}  =\frac{1}{1/\Delta t + \Delta t + (1-\Delta t)\eta^{k+1}} \left(\frac{v^{k}}{\Delta t} -   (1-\Delta t)\nabla f(x^k)\right)$
		\State $\xi^{k+1}=\begin{cases}
		     \frac{-\frac{\mathcal{E}^k}{\Delta t} + \sqrt{(\frac{\mathcal{E}^k}{\Delta t})^2+4\eta^{k+1}\left\|z^{k+1}\right\|^2\frac{\xi^k\mathcal{E}^{k-1}}{\Delta t}}}{2\eta^{k+1}\left\|z^{k+1}\right\|^2}, \quad \|z^{k+1}\| \neq 0\\
             \frac{ \xi^k \mathcal E^{k-1}}{\mathcal E^{k}}, \text{ otherwise }
		\end{cases}$
		\State $y^{k+1} = y^k + \Delta t((1-\eta^{k+1})z^{k+1}\xi^{k+1}-\xi^{k+1}\nabla f(x^k))$
		\State $x^{k+1} = x^k  +(1-\xi^{k+1} )(y^k - x^k) + \Delta t(z^{k+1}\xi^{k+1} +y^{k+1}- y^k$)
        \State $v^{k+1}=y^{k+1}-x^{k+1}$
		\EndFor
		\State \Return $x^N$
	\end{algorithmic}
\end{algorithm}


The positivity of the multiplier and the energy stable property follow directly from \eqref{eq2_29} and \eqref{eq2_22d}.

\begin{theorem}\label{th2_3}
	The introduced global Lagrange multiplier $\xi^{k}$ in  Algorithm \ref{alg:VOS_scheme}  is positive for $k \geq 0$. Furthermore, Algorithm \ref{alg:VOS_scheme} is unconditionally energy stable in the sense that the modified pseudo-energy
	\begin{equation}\label{eq2_23}
		\tilde{\mathcal{E}}^{k+1} - \tilde{\mathcal{E}}^{k} = -\Delta t \eta^{k+1} \left\| y^{k+1} - x^{k+1} \right\|^2 \leq 0, \quad \forall \Delta t > 0,k\geq0.
	\end{equation}
\end{theorem}
where \(\tilde{\mathcal{E}}^{k+1} = \xi^{k+1} \mathcal{E}(x^k,y^k)\).

As discussed in Section \ref{section2_1}, we provide a relaxed version and an adaptive version of the LLM-VOS scheme, namely RVOS and ARVOS methods, as in Algorithm \ref{alg:RVOS_scheme} and Algorithm \ref{alg:ARVOS_scheme}, respectively. The relaxed and adaptive variants are unconditionally energy stable as stated in the following theorem.

\begin{theorem}\label{thm:RVOSenergy_decay}
	The Lagrange multiplier $\xi^{k}$ in Algorithm \ref{alg:RVOS_scheme} and Algorithm \ref{alg:ARVOS_scheme} are positive for $k \geq 0$. Furthermore, Algorithm \ref{alg:RVOS_scheme} and Algorithm \ref{alg:ARVOS_scheme}  are unconditionally energy stable in the sense that the modified pseudo-energy is nonincreasing:
	\begin{equation*}
		\tilde{\mathcal{E}}^{k+1} \leq \tilde{\mathcal{E}}^k \leq \cdots \leq   \tilde{\mathcal{E}}^0, \quad  k \geq 0,
	\end{equation*}
    where \(\tilde{\mathcal{E}}^{k+1} := \xi^{k+1} \mathcal{E}(x^k, y^k)\) is nonnegative for all $k \geq 0$.
\end{theorem}

\begin{algorithm}
	\caption{ The Relaxed VOS (RVOS) Scheme}
	\label{alg:RVOS_scheme}
	\begin{algorithmic}[1]
		\State \textbf{Inputs:} $x^0, y^0 \in \mathbb{R}^n, \Delta t > 0, \xi^0 = 1$
        \State $x^{-1} = x^0, y^{-1} = y^0, v^0 = x^0 - y^0$
		\For{$k = 0, 1, \dots, N - 1$}
		\State $z^{k+1}  =\frac{1}{1/\Delta t + \Delta t + (1-\Delta t)\eta^{k+1}} \left(\frac{v^{k}}{\Delta t} -   (1-\Delta t)\nabla f(x^k)\right)$
		\State $\widetilde{\xi}^{k+1}=\begin{cases}
		     \frac{-\frac{\mathcal{E}^k}{\Delta t} + \sqrt{(\frac{\mathcal{E}^k}{\Delta t})^2+4\eta^{k+1}\left\|z^{k+1}\right\|^2\frac{\xi^k\mathcal{E}^{k-1}}{\Delta t}}}{2\eta^{k+1}\left\|z^{k+1}\right\|^2}, \quad \|z^{k+1}\| \neq 0\\
             \frac{ \xi^k \mathcal E^{k-1}}{\mathcal E^{k}}, \text{ otherwise }
		\end{cases}$
		\State $y^{k+1} = y^k + \Delta t((1-\eta^{k+1})z^{k+1}\widetilde{\xi}^{k+1}-\widetilde{\xi}^{k+1}\nabla f(x^k))$
		\State $x^{k+1} = x^k  +(1-\widetilde{\xi}^{k+1})(y^k - x^k) + \Delta t(z^{k+1}\widetilde{\xi}^{k+1} +y^{k+1}- y^k)$
        \State $v^{k+1}=y^{k+1}-x^{k+1}$
		\State $\xi^{k+1}=\min \left \{ \frac{ \xi^{k}\mathcal{E}^{k-1} } {\mathcal{E}^{k}}, \frac{\mathcal{E}^{k+1}} {\mathcal{E}^{k} } \right\}$
		\EndFor
		\State \Return $x^N$
	\end{algorithmic}
\end{algorithm}

\begin{algorithm}
	\caption{ The Adaptive RVOS (ARVOS) Scheme}
	\label{alg:ARVOS_scheme}
	\begin{algorithmic}[1]
		\State \textbf{Inputs:} $x^0, y^0 \in \mathbb{R}^n, \Delta t>0, \xi^0=1, \rho,\gamma,\Delta t_{min} > 0$
        \State $x^{-1} = x^0, y^{-1} = y^0, v^0 = x^0 - y^0$
		\For{$k = 0, 1, \dots, N - 1$}
		\If{$\xi^k<\gamma \ and\  \Delta t^k>\Delta t_{min}$}
			\State $\Delta t^{k+1}=\max\{\xi^k\Delta t^k,\Delta t_{min}\}$
		\Else
			\State $\Delta t^{k+1}=\rho\Delta t^{k}$
		\EndIf
		\State $\Delta t=\Delta t^{k+1}$
		\State $z^{k+1}  =\frac{1}{1/\Delta t + \Delta t + (1-\Delta t)\eta^{k+1}} \left(\frac{v^{k}}{\Delta t} -   (1-\Delta t)\nabla f(x^k)\right)$
		\State $\widetilde{\xi}^{k+1}=\begin{cases}
		     \frac{-\frac{\mathcal{E}^k}{\Delta t} + \sqrt{(\frac{\mathcal{E}^k}{\Delta t})^2+4\eta^{k+1}\left\|z^{k+1}\right\|^2\frac{\xi^k\mathcal{E}^{k-1}}{\Delta t}}}{2\eta^{k+1}\left\|z^{k+1}\right\|^2}, \quad \|z^{k+1}\| \neq 0\\
             \frac{ \xi^k \mathcal E^{k-1}}{\mathcal E^{k}}, \text{ otherwise }
		\end{cases}$
		\State $y^{k+1} = y^k + \Delta t((1-\eta^{k+1})z^{k+1}\widetilde{\xi}^{k+1}-\widetilde{\xi}^{k+1}\nabla f(x^k))$
		\State $x^{k+1} = x^k  +(1-\widetilde{\xi}^{k+1})(y^k - x^k) + \Delta t(z^{k+1}\widetilde{\xi}^{k+1} +y^{k+1}- y^k)$
        \State $v^{k+1}=y^{k+1}-x^{k+1}$
		\State $\xi^{k+1}=\min \left \{ \frac{ \xi^{k}\mathcal{E}^{k-1} } {\mathcal{E}^{k}}, \frac{\mathcal{E}^{k+1}} {\mathcal{E}^{k} } \right\}$
		\EndFor
		\State \Return $x^N$
	\end{algorithmic}
\end{algorithm}


\section{Convergence analysis}
In this section, we give the convergence analysis for our proposed schemes. We show detailed proofs for the SOF schemes and sketch the idea for the VOS schemes since rigorous proofs are direct extension.

With the modified pseudo-energy
$$\tilde{\mathcal{E}}^{k+1} := \xi^{k+1} \mathcal{E}(x^k, v^k), \quad k \geq 0,$$ we have the following convergence results regarding the convergence of sequence $\{v^{k}\}_{k\in \mathbb{N}}$ and $\{x^{k}\}_{k\in \mathbb{N}}$ inspired by \cite{CHEN2023111872}.

\begin{theorem}[Sequential convergence of $\{v^k\}_{k\in \mathbb{N}}$]\label{th3_1}
For the sequence of $\{v^k\}_{k\in \mathbb{N}}$ generated by Algorithm \ref{alg:SOF_scheme} or Algorithm \ref{alg:VOS_scheme}:
\begin{equation}\label{eq:sum_etav}
	\sum_{k=1}^{\infty}\eta^k\left\| v^{k} \right\|^2<+\infty.
\end{equation}

\noindent\textbf{Case 1 (Subsequential convergence)}:  if the damping coefficients $\{\eta^{k}\}_{k\in \mathbb{N}}$ satisfy
\begin{equation*}\label{eq2_12_2}
	\eta^k \geq \frac{\omega}{k}, \quad \text{ for some } \omega > 0,
\end{equation*}
then there exists a subsequence of $\left\{v^{k}\right\}_{k\in \mathbb{N}}$ converges to zero.

\noindent\textbf{Case 2 (Full sequential convergence)}: if the damping coefficients $\{\eta^{k}\}_{k\in \mathbb{N}}$ satisfy 
\begin{equation*}\label{eq2_12b_2}
  \eta^k \geq \epsilon_0  ,\quad \text{ for some } \epsilon_0 >0, 
\end{equation*}
then the sequence $\left\{v^{k}\right\}_{k\in \mathbb{N}}$ converges to zero.
\end{theorem}

\begin{proof}
We prove the convergence for $\{v^k\}_{k\in \mathbb{N}}$ generated by Algorithm \ref{alg:SOF_scheme}. For Algorithm \ref{alg:VOS_scheme}, the same argument follows using  $v^{k+1}=y^{k+1}-x^{k+1}$ and \eqref{eq2_22d}.

Rearranging \eqref{eq2_5c} we get
\begin{equation}\label{eq2_14_1}
	 \eta^{k+1} \left\| v^{k+1} \right\|^2=\frac{\tilde{\mathcal{E}}^{k}-\tilde{\mathcal{E}}^{k+1}}{\Delta t}\quad \forall~ \ k\geq0.
\end{equation}
For any $n\in\mathbb{N}$, summing over \eqref{eq2_14_1} for $k=0...,n-1$ and using $\tilde{\mathcal{E}}^{n} > 0$ , we get
\begin{equation*}\label{eq2_15_1}
	\sum_{k=1}^{n}\eta^k\left\| v^{k} \right\|^2=\frac{\tilde{\mathcal{E}}^{0}-\tilde{\mathcal{E}}^{n}}{\Delta t} \leq \frac{1}{\Delta t}\tilde{\mathcal{E}}^{0}, \quad \forall~ n \geq 1.
\end{equation*}
Taking $n \to \infty$, we have \eqref{eq:sum_etav}.

\noindent\textbf{Proof for Case 1:} We verify that $\liminf_{k\to\infty} \|v_k\| = 0$ by contradiction. Suppose $\liminf_{k\to\infty} \|v_k\| = \delta > 0$. Then there exists $k_0 \in \mathbb{N}$ such that for any $k > k_0$, $\|v_k\| \ge \delta.$
As  a result, 
$$
\sum_{k = k_0}^{\infty} \eta^k \|v^k\|^2 \geq \delta^2 \sum_{k = k_0}^{\infty} \eta^k  \geq  \omega \delta^2 \sum_{k = k_0}^{\infty} \frac{1}{k} ,
$$
which contradicts \eqref{eq:sum_etav}. Therefore, $\lim \inf_{k \to \infty} \|v_k\| = 0$ and we have the desired result.

\noindent\textbf{Proof for Case 2:}
By \eqref{eq:sum_etav} and the choice on $\{\eta^k\}_k$, we conclude that
\begin{equation*}
0 \leq  \lim \inf_{k \to \infty}\left\| v^{k} \right\|^2  \leq  \lim \sup_{k \to \infty}\left\| v^{k} \right\|^2  \leq  \frac{1}{\epsilon_0} \lim_{k \to \infty} \eta^k \left\| v^{k} \right\|^2  =0.
\end{equation*}
Therefore, $\{v^k\}_{k\in \mathbb{N}}$ converges to zero. 

\end{proof}




To prove the convergence of the sequence $\{x^k\}_{k\in \mathbb{N}}$, we further assume that $f$ is coercive:
$$
f(x) \to +\infty, \quad \text{ if } \|x\| \to \infty.
$$

\begin{theorem}[Uniform boundedness of $\{x^k\}_{k\in \mathbb{N}}$ in Algorithm \ref{alg:SOF_scheme}]\label{thm:boundedness_x} Suppose that $f$ is coercive and the Lagrange multiplier $\{\xi^k\}_{k\in \mathbb{N}}$ satisfy
$$\xi^k \geq \delta_0, \quad \text{ for some } \delta_0 > 0.$$
Then the sequence of $\{x^k\}_{k\in \mathbb{N}}$ generated by Algorithm \ref{alg:SOF_scheme} is uniformly bounded.

\begin{proof}
   By \eqref{eq2_5c}, the modified pseudo-energy is nonincreasing and uniformly bounded
\begin{equation*}\label{eq:tildeEdecay}
	 0 \leq \tilde{\mathcal{E}}^{k+1} \leq \tilde{\mathcal{E}}^{k} \leq \cdots \leq   \tilde{\mathcal{E}}^{0}.
\end{equation*}

Noted that by the assumption that $\xi^{k+1}$ has a uniform lower bound,
$$
0 < f(x^k) \leq   f(x^k) + \frac{1}{2}\|v^k\|^2 = \frac{\tilde{\mathcal{E}}^{k+1}}{\xi^{k+1}} \leq \frac{\tilde{\mathcal{E}}^{0}} {\delta_0} < +\infty.
$$
Therefore $\{x^k\}_{k\in \mathbb{N}}$ is uniformly bounded since $f$ is coercive.

\end{proof}

\end{theorem}

\begin{theorem}[Uniform boundedness of $\{x^k\}_{k\in \mathbb{N}},\{y^k\}_{k\in \mathbb{N}}$ in Algorithm \ref{alg:VOS_scheme}]\label{thm:boundedness_VOS} Suppose that $f$ is coercive and the Lagrange multiplier $\{\xi^k\}_{k\in \mathbb{N}}$ satisfy
$$\xi^k \geq \delta_0, \quad \text{ for some } \delta_0 > 0.$$
Then the sequences of $\{x^k\}_{k\in \mathbb{N}},\{y^k\}_{k\in \mathbb{N}}$ generated by Algorithm \ref{alg:VOS_scheme} is uniformly bounded. 
\end{theorem}

\begin{proof}
    The boundedness of $\{x^k\}_{k\in \mathbb{N}}$ follows by \eqref{eq2_22d} and the proof of Theorem \ref{thm:boundedness_x}. Since $f>0$, we have
    $$
    \dfrac{1}{2}\|v^k\|^2\leq f(x^k)+\dfrac{1}{2}\|v^k\|^2\leq\frac{\tilde{\mathcal{E}^0}}{\delta_0},
    $$
    which implies $\{v^k\}_{k\in \mathbb{N}}$ is uniformly bounded. The boundedness of $\{y^k\}_{k\in \mathbb{N}}$ follows from $y^k=x^k+v^k$.
\end{proof}

\begin{rem} \label{xi_uniform lower bound}
Theorem \ref{thm:boundedness_x} and Theorem \ref{thm:boundedness_VOS} require a uniform positive lower bound for $\{\xi^k\}_{k\in \mathbb{N}}$. 
We can modify \eqref{eq2_5c} and \eqref{eq2_22d} to
\begin{equation*}
    \dfrac{\widehat{\mathcal{E}}^{k+1}-\widehat{\mathcal{E}}^{k}}{\Delta t}=-\eta^{k+1}\|v^{k+1}\|^2,\label{modified_energy}
\end{equation*}
where $\widehat{\mathcal{E}}^{k+1}+\delta^{k+1}=\xi^{k+1}(\mathcal{E}^k+\delta^{k+1})$. Then we get $\xi^{k+1}$ satisfies the following algebraic equation:
$$
F(\xi^{k+1})=\eta^{k+1}\|z^{k+1}\|^2(\xi^{k+1})^2+\dfrac{\mathcal{E}^k+\delta^{k+1}}{\Delta t}\xi^{k+1}-\dfrac{\delta^{k+1}+\widehat{\mathcal{E}}^k}{\Delta t}=0.
$$
We set $\delta^0 \geq 0$ and $\delta^0 \neq 1$, and compute
$$
\delta^{k+1}=\max\left\{\dfrac{\eta^{k+1}\|z^{k+1}\|^2\delta_0^2\Delta t+\mathcal{E}^k\delta_0-\widehat{\mathcal{E}}^k}{1-\delta_0},0\right\}.
$$
It is straightforward to verify that $F(\delta_0)\leq0$ and therefore $\xi^{k+1} \geq \delta_0$.

For energy stability, by \eqref{modified_energy} we get
$$
\widehat{\mathcal{E}}^{k+1}-\widehat{\mathcal{E}}^{k}=- \eta^{k+1}\|v^{k+1}\|^2 \Delta t\leq0,
$$
where $\widehat{\mathcal{E}}^{k+1}=\widetilde{\mathcal{E}}^{k+1}+(\xi^{k+1}-1)\delta^{k+1}$ is an approximation of $\widetilde{\mathcal{E}}^{k+1}$.
\end{rem}

As a result of Theorem \ref{thm:boundedness_x}, there exists a subsequence of $\{x^k\}_{k\in \mathbb{N}}$ converges by the Bolzano–Weierstrass theorem. We next show that the limit points of $\{x^k\}_{k\in \mathbb{N}}$ form a compact and connected set.
\begin{theorem}[Properties of the cluster-point set in in Algorithm \ref{alg:SOF_scheme}]\label{thm:cluster_set}
Suppose that $f$ is coercive and if the Lagrange multiplier $\{\xi^k\}_{k\in \mathbb{N}}$  and the damping coefficients $\{\eta^{k}\}_{k\in \mathbb{N}}$ satisfy
$$\xi^k \geq \delta_0, \, \eta^{k+1} \geq \epsilon_0  \quad  \text{ for some } \delta_0 > 0,  \epsilon_0 >0, \quad \forall~k\geq0.$$
For the sequence of $\{x^k\}_{k\in \mathbb{N}}$  generated by Algorithm \ref{alg:SOF_scheme}, define the set 
$$
\Omega
:=
\left\{
\bar{x}\in\mathbb{R}^N:
\text{ there exists a subsequence }
\{x^{k_j}\}_{j\geq1}
\text{ such that }
x^{k_j} \to\bar{x} \text{ as } j \to \infty 
\right\}.
$$
Then $\Omega$ is nonempty, connect and compact. Moreover,
$$
\operatorname{dist}(x^k,\Omega)\to0
\qquad\text{as }k\to\infty, 
$$
where $\operatorname{dist}(x,\Omega):=\inf_{y\in\Omega}|x-y|.$
\end{theorem}

\begin{proof}
As a consequence of Theorem \ref{thm:boundedness_x}, $\{x^k\}_{k\in \mathbb{N}}$
is bounded and there exists a convergent subsequence. Therefore, $\Omega$ is nonempty.

We next show that $\Omega$ is compact. For every $\bar{x}\in\Omega$, there exists a subsequence
$\{x^{k_j}\}_{j\geq1}$ such that $x^{k_j}\to\bar{x}$ as $j \to \infty$.  Since $\{x^k\}_{k\in \mathbb{N}}$ is bounded, there exists $M>0$ such that
$$
\|x^{k_j}\|\leq M,
\qquad \forall~ j\geq1.
$$
Therefore, $\|\bar x\| \leq M$ and $\Omega$ is bounded. To prove that $\Omega$ is closed, consider $\bar x \in \mathbb{R}^N$ such that $\bar{x}^m \to \bar{x}$ for $\{\bar{x}^m\}_{m\geq1}\subset\Omega$. Since $\bar{x}^m\in\Omega$, for each $m\geq1$ we can choose an index $k_m>k_{m-1}$ such that
$$
\|x^{k_m}-\bar{x}^m\|<\frac{1}{m}.
$$
Then
$$
\|x^{k_m}-\bar{x}\|
\leq
\|x^{k_m}-\bar{x}^m\|
+
\|\bar{x}^m-\bar{x}\|
\to0, \quad \text { as } m \to \infty.
$$
Therefore,
$ x^{k_m}\to\bar{x}$, which implies $\bar{x}\in\Omega$. Hence, $\Omega$ is closed. Since $\Omega \subseteq \mathbb{R}^N$ is closed and bounded, it is compact.

We now prove that $\operatorname{dist}(x^k,\Omega)\to0$ by contradiction. Suppose that there exists
$\varepsilon>0$ and a subsequence $\{x^{k_j}\}_{j\geq1}$ such that
$$
\operatorname{dist}(x^{k_j},\Omega)\geq\varepsilon
\qquad \forall~ j \geq 1.
$$
Since $\{x^{k_j}\}_{j\geq1}$ is bounded, it has a convergent
subsequence, denoted again by $\{x^{k_j}\}_{j\geq1}$, such that $x^{k_j}\to x^\ast$ for some $x^\ast\in\mathbb{R}^N$. By definition, $x^\ast\in\Omega$.
Therefore,
$$
\operatorname{dist}(x^{k_j},\Omega)
\leq
\|x^{k_j}-x^*\|
\to0,
$$
which contradicts $ \operatorname{dist}(x^{k_j},\Omega)\geq\varepsilon$. Hence, $\operatorname{dist}(x^k,\Omega)\to0.$

It remains to prove that $\Omega$ is connected. Suppose, by
contradiction, that $\Omega$ is disconnected. Then there exist two
nonempty disjoint compact sets $\Omega_1$ and $\Omega_2$ such that $\Omega = \Omega_1\cup\Omega_2.$
Since $\Omega_1$ and $\Omega_2$ are disjoint and compact, their
distance is strictly positive:
$$
d
:=
\operatorname{dist}(\Omega_1,\Omega_2)
:=
\inf\left\{
|y-z|:
y\in\Omega_1,\ z\in\Omega_2
\right\}> 0.$$
Choose
$$
0<\delta<\frac{d}{4},
$$
and define the open neighborhoods
$$
U_i
:=
\left\{
x\in\mathbb{R}^N:
\operatorname{dist}(x,\Omega_i)<\delta
\right\},
\qquad i=1,2.
$$
Then $U_1$ and $U_2$ are disjoint. 
Since $\operatorname{dist}(x^k,\Omega)\to0$, there exists $K_1\geq0$ such that
$$
\operatorname{dist}(x^k,\Omega)<\delta,
\qquad k\geq K_1,
$$
which implies
$$
x^k\in U_1\cup U_2,
\qquad k\geq K_1.
$$

On the other hand, by Theorem \ref{th3_1},
$$
\|x^{k+1}-x^k\|\to0, \quad \text{ as } k \to +\infty,
$$
there exists $K_2\geq0$ such that
$$
\|x^{k+1}-x^k\|<d-2\delta,
\qquad k\geq K_2.
$$
Let $K:=\max\{K_1,K_2\}.$ Since both $\Omega_1$ and $\Omega_2$ are nonempty subsets of the
cluster-point set $\Omega$, the sequence $\{x^k\}_{k\geq0}$ visits both $U_1$ and $U_2$ infinitely often. There exists $p\geq K$ such
that
$$
x^p\in U_1,
\quad
x^{p-1}\in U_2, \quad \text{ and } \|x^p-x^{p-1}\|<d-2\delta.
$$
Using the triangle inequality , we get
$$
\begin{aligned}
    \|x^p-x^{p-1}\| &= \|x^p - \bar x + \bar x - \bar y + \bar y - x^{p-1}\| \\
    &\geq | \|\bar x - \bar y\|  -  \|x^p - \bar x  + \bar y - x^{p-1} \|    | \\
    &\geq |d- (\operatorname{dist}(x^p,\Omega_1) +  \operatorname{dist}(x^{p-1},\Omega_2))| \\
    &> d-2\delta,
\end{aligned}
$$
where $\bar x \in \Omega_1$ and $\bar y \in \Omega_2$ such that $\|x^p - \bar x\| = \operatorname{dist}(x^p,\Omega_1) $ and  $\|x^{p-1} - \bar y\| = \operatorname{dist}(x^{p-1},\Omega_2) $. This contradicts $\|x^p-x^{p-1}\|<d-2\delta.$
Hence, $\Omega$ is connected.
\end{proof}

Similar results can be obtained for Algorithm \ref{alg:VOS_scheme}. As a result of Theorem \ref{thm:boundedness_VOS}, there exist subsequences of $\{x^k\}_{k\in \mathbb{N}},\{y^k\}_{k\in \mathbb{N}}$ converge by the Bolzano–Weierstrass theorem and their limit points form a compact and connected set.

\begin{theorem}[Properties of the cluster-point set in Algorithm \ref{alg:VOS_scheme}]\label{thm:cluster_set_VOS}
Suppose that $f$ is coercive, $0<\Delta t<1$ and if the Lagrange multiplier $\{\xi^k\}_{k\in \mathbb{N}}$  and the damping coefficients $\{\eta^{k}\}_{k\in \mathbb{N}}$ satisfy
$$\xi^k \geq \delta_0, \, \eta^{k+1} \geq \epsilon_0  \quad  \text{ for some } \delta_0 > 0,  \epsilon_0 >0, \quad \forall~k\geq0.$$
For the sequences of $\{x^k\}_{k\in \mathbb{N}}$, $\{y^k\}_{k\in \mathbb{N}}$  generated by Algorithm \ref{alg:VOS_scheme}, define the set 
$$
\Omega_{x}
:=
\left\{
\bar{x}\in\mathbb{R}^N:
\text{ there exists a subsequence }
\{x^{k_j}\}_{j\geq1}
\text{ such that }
x^{k_j} \to\bar{x} \text{ as } j \to \infty 
\right\}.
$$
$$
\Omega_{y}
:=
\left\{
\bar{y}\in\mathbb{R}^N:
\text{ there exists a subsequence }
\{y^{k_j}\}_{j\geq1}
\text{ such that }
y^{k_j} \to\bar{y} \text{ as } j \to \infty 
\right\}.
$$
Then $\Omega_{x},\Omega_{y}$ are nonempty, connect and compact. Moreover,
$$
\Omega_{x}=\Omega_{y},\quad\operatorname{dist}(x^k,\Omega)\to0,\quad\operatorname{dist}(y^k,\Omega)\to0
\qquad\text{as }k\to\infty, 
$$
where $\operatorname{dist}(x,\Omega):=\inf_{y\in\Omega}|x-y|.$
\end{theorem}

\begin{proof}
    It suffices to verify that $\Omega_{x} = \Omega_{y}$, the remaining arguments follows directly using the proof of Theorem \ref{thm:cluster_set}. Suppose $\bar x \in \Omega_x$ and $\bar y \in \Omega_y$, then there exist
    $$
    x^{k_{j_1}}\to \bar x, \quad  y^{k_{j_2}} \to \bar y, \quad \text{ as } j_1, \, j_2 \to \infty.
    $$
    By Theorem \ref{th3_1}, $v^k \to 0$ implies
    $$
    y^{k_{j_1}} = v^{k_{j_1}} + x^{k_{j_1}} \to \bar x, \quad x^{k_{j_2}} = y^{k_{j_2}} - v^{k_{j_2}} \to \bar y, \quad \text{ as } j_1, \, j_2 \to \infty.
    $$
    Therefore, $\bar x \in \Omega_y$ and $\bar y \in \Omega_x$. We conclude that $\Omega_x \subseteq \Omega_y$ and $\Omega_y \subseteq \Omega_x$.


\end{proof}

Finally, we show that $\{x^k\}_{k \in \mathbb{N}}$ converges to stationary points.

\begin{theorem}[Convergence to stationary points]\label{thm:stationary_conv}
Suppose that $f$ is coercive and continuously differentialble, and the linear operator $\mathcal L$ is bounded. Let $\{(x^k,v^k,\xi^k)\}_{k\in \mathbb{N}}$ be generated by Algorithm
\ref{alg:SOF_scheme} or Algorithm
\ref{alg:VOS_scheme} (with $0<\Delta t<1$). Suppose that the  Lagrange multiplier $\{\xi^k\}_{k\in \mathbb{N}}$  and the damping coefficients $\{\eta^{k}\}_{k\in \mathbb{N}}$ satisfy
$$\xi^k \geq \delta_0, \, \eta^{k+1} \leq \eta_{\max}  \quad  \text{ for some } \delta_0 > 0,  \, \eta_{\max}>0, \quad \forall~k\geq0.$$

\noindent
\textbf{Case 1 (Subsequential convergence):}
Suppose that, for some $\omega>0$,
\begin{equation}\label{eq:case1_eta_stationary}
    \eta^k\geq\frac{\omega}{k},
    \qquad k\geq1.
\end{equation}
Then there exist a subsequence $\{x^{k_j}\}_{j\geq1}$ and a point
$x^*\in\mathbb{R}^N$ such that
$$
    x^{k_j}\to x^*,
    \quad
    \nabla f(x^{k_j})\to0, \text{ as } j \to \infty \, \text{ and } \nabla f(x^*)=0.
$$
\noindent
\textbf{Case 2 (Full sequence  convergence):}
Suppose that, for some $\epsilon_0>0$,
\begin{equation*}\label{eq:case2_eta_stationary}
    \eta^k\geq\epsilon_0,
    \qquad k\geq1.
\end{equation*}
Then
$$
    \nabla f(x^k)\to0, \quad \text{ as }  k \to \infty.
$$
\end{theorem}

\begin{proof}
We first prove the results for Algorithm \ref{alg:SOF_scheme}. 

\noindent\textbf{Proof for Case 1:} Define $c_k:=\|v^k\|^2+\|v^{k+1}\|^2. $ By \eqref{eq:sum_etav} and \eqref{eq:case1_eta_stationary},
\begin{align*}
    \sum_{k=1}^{\infty} \frac{c_k}{k+1}
    &=
  \sum_{k=1}^{\infty}    \frac{\|v^k\|^2 + \|v^{k+1}\|^2}{k+1}
    \leq
   \sum_{k=1}^{\infty}   \frac{\|v^k\|^2}{k}
    +
    \sum_{k=1}^{\infty}  \frac{\|v^{k+1}\|^2}{k+1}
    \\
    &\leq
    \frac{1}{\omega}
    \left(
        \eta^k\|v^k\|^2
        +
        \eta^{k+1}\|v^{k+1}\|^2
    \right) <\infty.
\end{align*}

Therefore, $\liminf_{k\to\infty}c_k=0$ and there exists a subsequence $\{k_j\}_{j\geq1}$ such that
$$
    v^{k_j}\to0,
    \qquad
    v^{k_j+1}\to0, \quad \text{ as } j \to \infty .
$$

Using \eqref{eq2_7} with
$ z^{k+1}= v^{k+1}/\xi^{k+1}$,  we obtain
\begin{equation}\label{eq:gradient_formula_stationary}
    \nabla f(x^k)
    =
    \frac{v^k}{\Delta t\,\xi^k}
    -
    \left(
        \frac{1}{\Delta t}I
        +
        S\Delta t\,\mathcal{L}
        +
        \eta^{k+1}I
    \right)
    \frac{v^{k+1}}{\xi^{k+1}}, \quad k \geq 0.
\end{equation}
Consequently,
\begin{align*}
    \|\nabla f(x^{k_j})\|
    &\leq
    \frac{1}{\Delta t\,\delta_0}
    \|v^{k_j}\|+
    \frac{1}{\delta_0}
    \left(
        \frac{1}{\Delta t}
        +
        S\Delta t\,\|\mathcal{L}\|
        +
        \eta_{\max}
    \right)
    \|v^{k_j+1}\| \to 0,  \quad \text{ as } j \to \infty.
\end{align*}

By Theorem \ref{thm:boundedness_x}, $\{x^k\}_{k\geq0}$ is bounded. As a result, the sequence $\{x^{k_j}\}_{j\geq1}$ has a convergent subsequence. With abuse of notation, we denote the convergent subsequense  as $\{x^{k_j}\}_{j\geq1}$. Let $x^*$ denote the point such that $x^{k_j}\to x^*$. $x^*$ is a stationary point of $f$ since $\nabla f$ is continuous,
$$
    \nabla f(x^*)
    =
    \lim_{j\to\infty}
    \nabla f(x^{k_j})
    =
    0.
$$

\noindent\textbf{Proof for Case 2:} Applying \eqref{eq:gradient_formula_stationary}, we obtain
\begin{align*}
    \|\nabla f(x^k)\|
    &\leq
    \frac{1}{\Delta t\,\delta_0}
    \|v^k\|+
    \frac{1}{\delta_0}
    \left(
        \frac{1}{\Delta t}
        +
        S\Delta t\,\|\mathcal{L}\|
        +
        \eta_{\max}
    \right)
    \|v^{k+1}\| \to 0,
\end{align*}
since both $\{v^k\}_{k \in \mathbb{N}}$ converges to zero by Theorem \ref{th3_1} (Case 2).

For Algorithm \ref{alg:VOS_scheme}, by \eqref{eq2_27} and $z^{k+1}=v^{k+1}/\xi^{k+1}$,
\[
\nabla f(x^k)
=
\frac{1}{1-\Delta t}
\left[
\frac{v^k}{\Delta t}
-
\left(
\frac{1}{\Delta t}
+\Delta t
+(1-\Delta t)\eta^{k+1}
\right)
\frac{v^{k+1}}{\xi^{k+1}}
\right].
\]
With $0< \Delta t < 1$, there exists $C>0$ independent of $k$ such that
$$
\|\nabla f(x^k)\|\leq C\bigl(\|v^k\|+\|v^{k+1}\|\bigr), \quad k \geq 0.
$$
The arguments follow in a similar way. For brevity, we skipped the detailed proof.
\end{proof}

By Theorem \ref{thm:stationary_conv}, every cluster point of $\{x^k\}_{k\geq0}$ is a stationary point of $f$. If, in addition, the stationary points of $f$ in the compact set containing $\{x^k\}_{k\geq0}$ are isolated, then the whole sequence converges to a stationary point $x^*$. Moreover, for the LLM-VOS scheme (Algorithm \ref{alg:VOS_scheme}), the convergence results to the stationary points can be applied to the sequence $\{y^k\}_{k \in \mathbb{N}}$.



\section{Numerical results}
We validate our proposed algorithms through benchmark problems on function optimization and numerical PDEs. The proposed methods escape local minima and achieve high accuracy during neural network training for solving PDEs. Furthermore, we show that our proposed methods exhibit strong robustness and effectiveness with respect to the choice of initial values and large learning rates, highlighting their promising potential for solving large-scale and high-dimensional problems.

\subsection{Function optimization}\label{section3_1}
In this section, we present several numerical experiments for function optimization, including the Rosenbrock function and Rastrigin function. In Section \ref{subsub1b}, we test the adaptive algorithms, ARSOF and ARVOS, and compared them with a number of other adaptive algorithms, such as ADAM and ARSAV \cite{shen2023}. In Section \ref{subsub1a}, we test the relaxed algorithms, RSOF and Rvos, and compare them with several methods based on the first-order gradient flow, such as GD, RLLM \cite{LIU2026114750}, RSAV \cite{shen2023} and so on.

\subsubsection{Rosenbrock function}\label{subsub1b}
Consider the two-dimensional (2D) Rosenbrock function
\begin{equation*}
	f(x, y) = (a - x)^2 + b(y - x^2)^2
\end{equation*}
which has no significant local extrema but is difficult to optimise; it is a classic benchmark function for testing the convergence and path-finding capabilities of optimization algorithms. This function has a global minimum at \((x, y) = (a, a^2)\), which is inside a long narrow, parabolic shaped flag valley. To find the valley is trivial, but to converge to the global minimum is usually difficult. We set $a=1$, $b=100$. All algorithms start with the initial point with coordinate $(-3,4)$ and terminate when $|f(x)-f(x^*)|$ is less than the preset precision of $\num{0.0001}$ or reaching the maximum number of iterations ($1,000$ iterations).

We test the ARSOF method (Algorithm \ref{alg:ARSOF_scheme}) and the ARVOS method (Algorithm \ref{alg:ARVOS_scheme}), and compare with ARSAV \cite{shen2023}, RLLM \cite{LIU2026114750}, ADAM \cite{kingma2014adam}, NAG \cite{JMLR:v17:15-084} and the gradient descent method. The parameters of the reference algorithms are set as suggested in the paper: NAG with \( \gamma = 0.9 \) and ADAM with \( \beta_1 = 0.9, \beta_2 = 0.999, \varepsilon = 10^{-8} \), as well as RLLM, RSAV with $S=1$ and $\mathcal{L}=\mathcal{I}$. We shall keep using these default settings in the next experiment. 

In Tables~\ref{tabX_location} and~\ref{tab3_2b}, we report the final error value $\|x - x^*\|_2$ and the numbers of iterations with respect to different step size $\Delta t$. ARSOF and ARVOS attain prescribed accuracy for the various initial learning rates, while most of the remaining algorithms do not converge within the maximum number of iteration (diverge or yield \texttt{NaN}).


Figure~\ref{fig3_2b} gives the corresponding modified pseudo-energy curves for ARSOF and ARVOS with $\Delta t=0.01$. For both methods, the pseudo-energy decreases monotonically, in accordance with Theorems~\ref{thm:RSOFenergy_decay} and~\ref{thm:RVOSenergy_decay}. We collected the searching paths generated with initial learning rate $\Delta t=0.001$ in Figure~\ref{fig3_2e}, where ARSOF and ARVOS progress along the Rosenbrock valley while other methods make limited movement from the starting point.

Figure~\ref{fig3_2c} illustrates how the Lagrange multiplier $\xi$ evolves. With an adaptive learning rate, $\xi$ is not required to equal $1$ but remains strictly positive with a positive lower bound, which is consistent with the analysis in \cite{LIU2026114750} and discussion in Remark \ref{xi_uniform lower bound}.


\begin{table}[!t]
	\renewcommand{\arraystretch}{1.1}
	\centering
	\caption{ The error $\|x - x^*\|_2$ using different methods with different initial learning rates for Rosenbrock function}\label{tabX_location}
	\begin{tabular}{c c c c c c }
		\hline
		Method $\backslash$ $\Delta t$ & $0.0001$ & $0.001$ & $0.01$  &  $0.1$ & $1$ \\
		\hline
		ARSOF & $\num{0.022286305674057125}$ & $\num{0.022085627583702295}$ & $\num{0.022045814577707802}$ & $\num{0.02227661985527798}$ & $\num{0.022138495402611598}$  \\
		\hline
		ARVOS & $\num{0.02212452376800236}$ & $\num{0.015061959586613353}$ & $\num{0.022401328913772268}$ & $\num{0.012988114002125392}$ & $\num{0.012468052118263672}$  \\
		\hline
		ARSAV & $\num{4.222798308723384}$ & $\num{4.0614528123097875}$ & $\num{4.214812778000525}$ & $\num{4.126765294682298}$ & $\num{4.209649256717768}$ \\
		\hline
		RLLM & $\num{4.2807426035769485}$ & $\num{3.299549172652342}$ & $\num{0.7295480535797394}$ & $\num{11.40175425099138}$ & $\num{11.40175425099138}$ \\
		\hline
		ADAM & $\num{4.982995347744638}$ & $\num{4.984293930885773}$ & $\num{4.971476376067657}$ & $\num{3.5904601082547907}$ & $\num{0.04327541434879419}$ \\
		\hline
		NAG & $\num{2.3896195464986905}$ & \texttt{NaN} & \texttt{NaN} & \texttt{NaN} & \texttt{NaN}\\
		\hline
		GD & $\num{4.275188033533277}$ & $\num{9.899494936611665}$ & $\num{11.40175425099138}$ & $\num{11.40175425099138}$ & $\num{11.40175425099138}$ \\
		\hline
	\end{tabular}
\end{table}

\begin{table}[!t]
		\renewcommand{\arraystretch}{1.1}
		\centering
		\caption{The number of iterations reaching the specified accuracy using ARSOF method (Algorithm \ref{alg:ARSOF_scheme}) and the ARVOS method (Algorithm \ref{alg:ARVOS_scheme}) with different initial learning rates for Rosenbrock function.}\label{tab3_2b}
		\begin{tabular}{c c c c c c c }
			\hline
			Method $\backslash$ $\Delta t$ & $0.0001$ & $0.001$ & $0.01$  &  $0.1$ & $1$ & Mean \\
			\hline
			ARSOF & $746$ & $733$ & $581$ & $866$ & $867$ & $758.6$ \\
			\hline
			ARVOS & $673$ & $658$ & $537$ & $291$ & $317$ & $495.2$ \\
			\hline
		\end{tabular}
\end{table}

\begin{figure}[htbp]
  \centering
  \makebox[\textwidth][c]{%
    \begin{minipage}[b]{0.64\textwidth}
      \centering

      \subfigure[ARSOF]{%
        \begin{minipage}[t]{0.485\linewidth}
          \centering
          \includegraphics[
            width=\linewidth,
            height=4.2cm
          ]{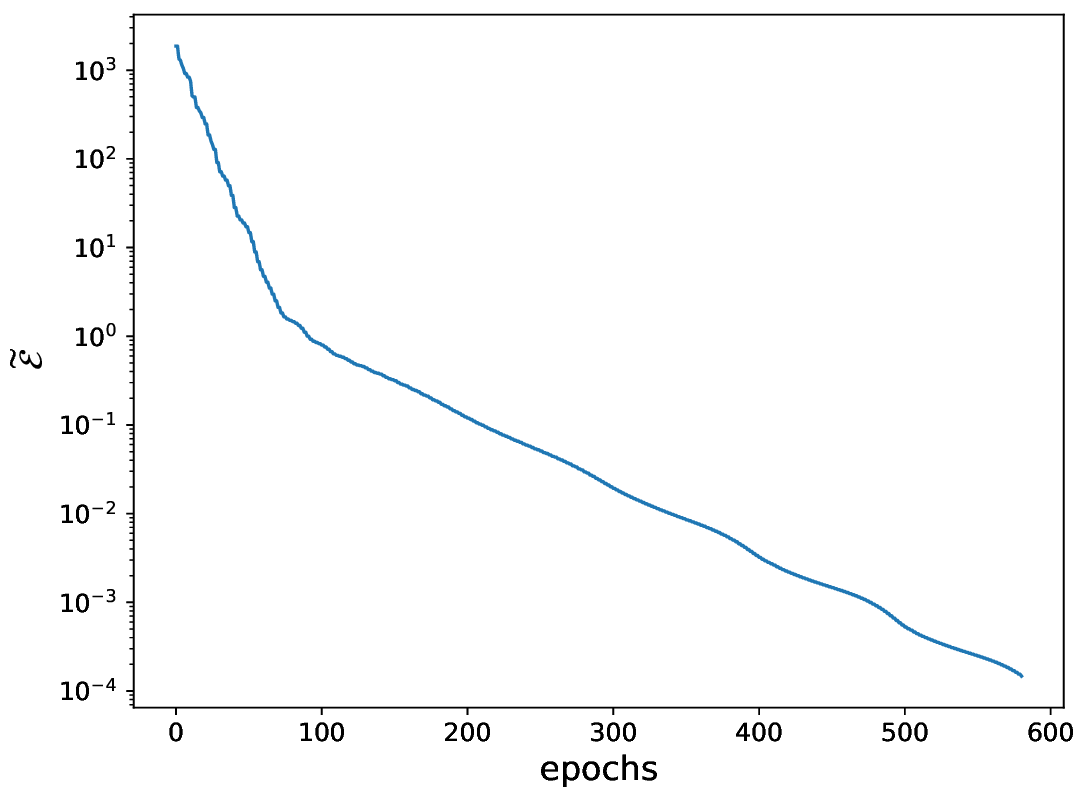}
        \end{minipage}%
      }%
      \hfill
      \subfigure[ARVOS]{%
        \begin{minipage}[t]{0.485\linewidth}
          \centering
          \includegraphics[
            width=\linewidth,
            height=4.2cm
          ]{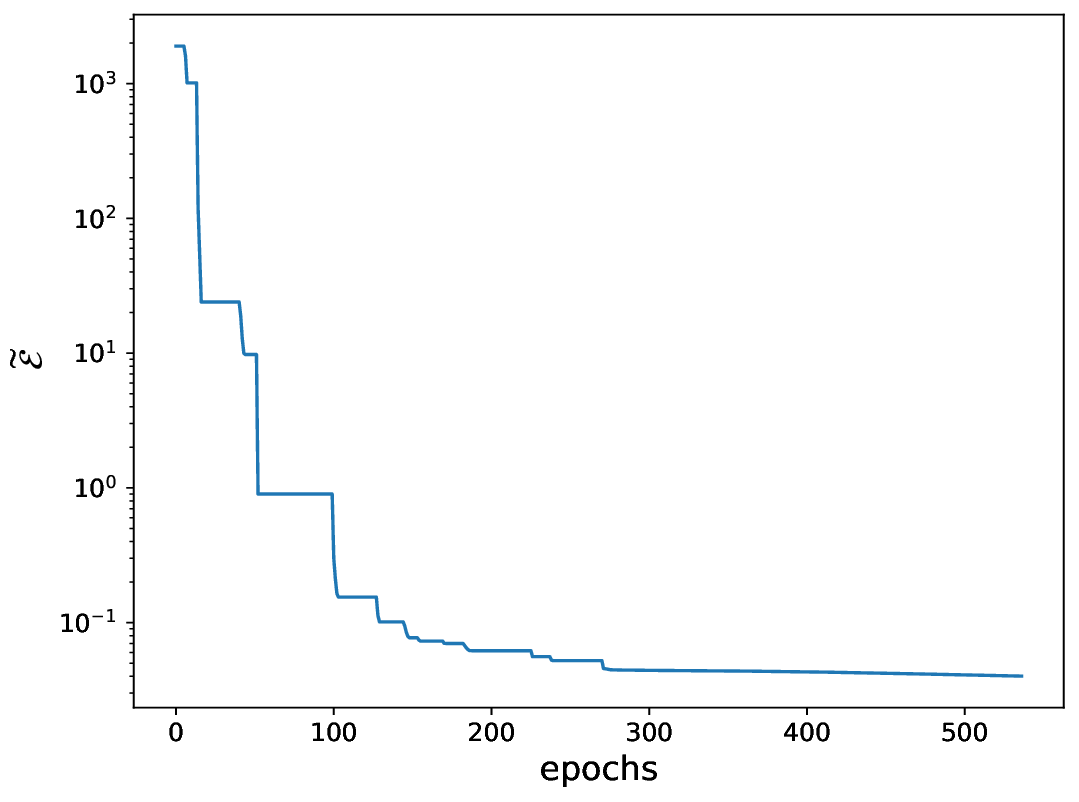}
        \end{minipage}%
      }

      \caption{Energy curves for Rosenbrock function with initial learning rate $\Delta t=0.01$}
      \label{fig3_2b}
    \end{minipage}%
    \hfill
    \begin{minipage}[b]{0.31\textwidth}
      \centering

      \includegraphics[
        width=\linewidth,
        height=4.2cm
      ]{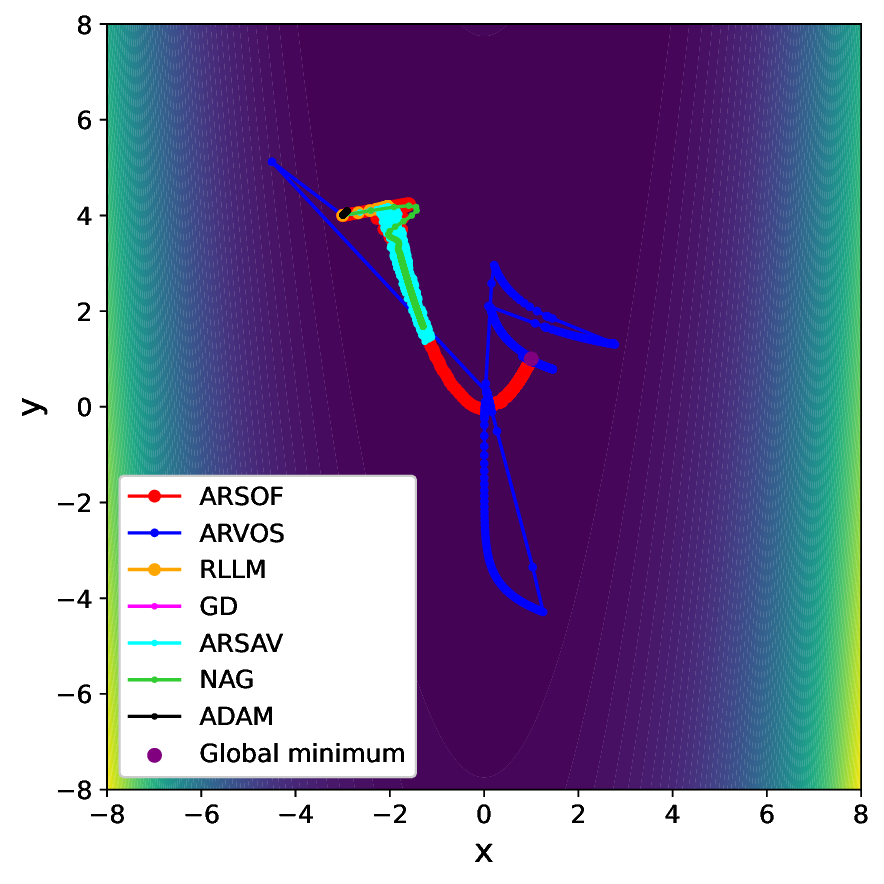}

      \caption{Search paths using different methods with (initial) learning rate $\Delta t=0.001$ for Rosenbrock function}
      \label{fig3_2e}
    \end{minipage}%
  }
\end{figure}



\begin{figure}[htbp]
	\centering
	\subfigure[ARSOF]{
		\begin{minipage}[c]{0.4\linewidth}
			\centering
			\includegraphics[scale=0.32]{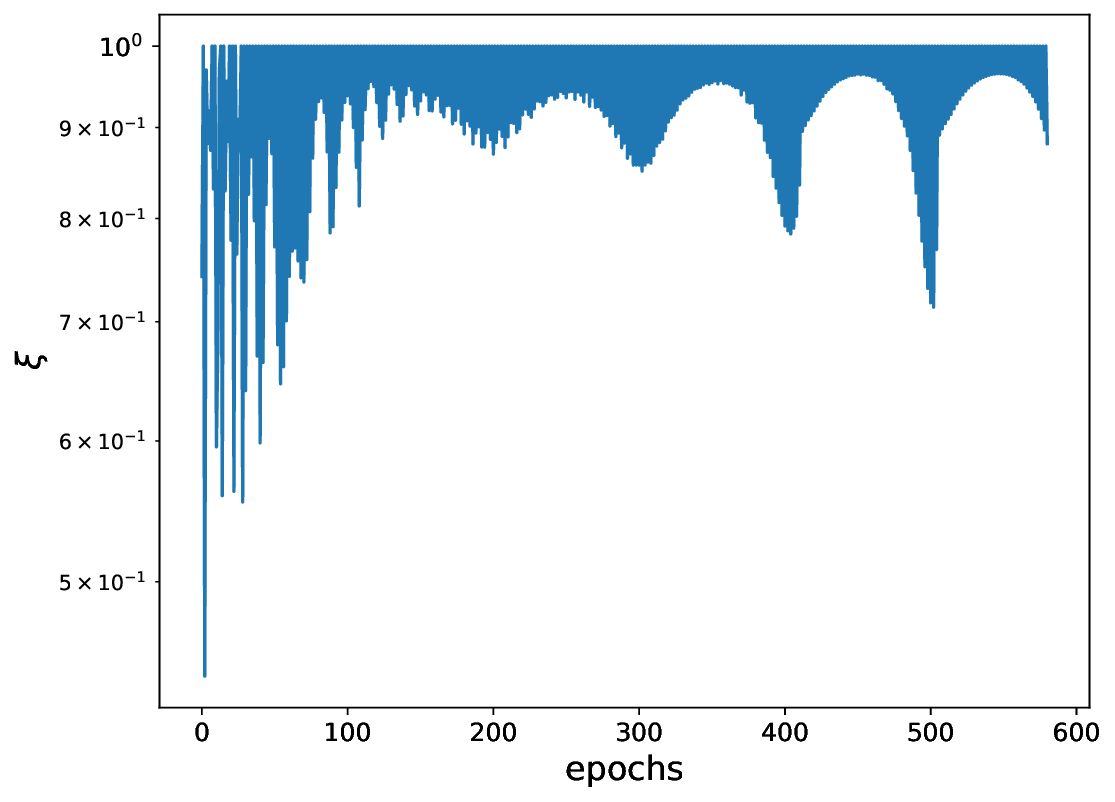}
	\end{minipage}}
	\subfigure[ARVOS]{
		\begin{minipage}[c]{0.4\linewidth}
			\centering
			\includegraphics[scale=0.32]{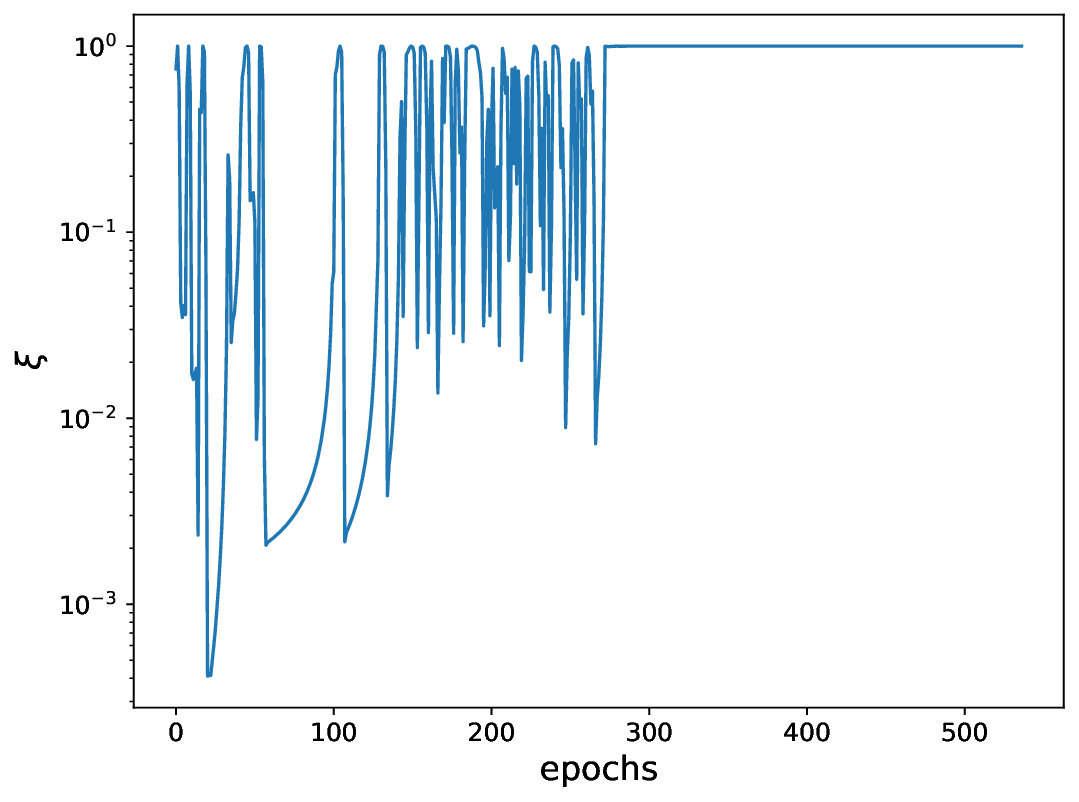}
	\end{minipage}}
	\caption{Variation of Lagrange multiplier $\xi$ for Rosenbrock function with initial learning rate $\Delta t=0.01$}\label{fig3_2c}
\end{figure}



\subsubsection{Rastrigin function}\label{subsub1a}
Consider the Rastrigin function
\begin{equation*}
	f(x) = f(\theta_1, \theta_2, \ldots, \theta_n) = \sum_{i=1}^n \theta_i^2 + 10n - 10 \sum_{i=1}^n \cos(2\pi\theta_i),
\end{equation*}
which has a large number of local optima and is highly prone to getting stuck in local optima. The function has a global minimum \( f^* = 0 \) located at \( x^* = (0, 0, \ldots, 0) \). For this example, all algorithms start with the initial point with coordinate  $(-3.5,3.5)$ and terminate when $|f(x)-f(x^*)|$ is less than the preset precision of $\num{0.0001}$ or reaching the maximum number of iterations ($1,0000$ iterations).

We test the RSOF method (Algorithm \ref{alg:RSOF_scheme}) and the RVOS method (Algorithm \ref{alg:RVOS_scheme}) and compare with RSAV \cite{shen2023} and other reference optimization algorithms, using their default settings. We report the final loss values and the numbers of iterations required to reach a prescribed accuracy with respect to different learning rates in Tables~\ref{tab3_1a} and~\ref{tab3_1b}, respectively. Both RSOF and RVOS converge to the optimal function value for all tested learning rates, whereas the remaining methods stagnate at large loss values and, in most cases, fail to satisfy the accuracy criterion within the prescribed maximal iteration number. Furthermore, the remaining methods became trapped in local minima ($|f(x)-f^*|=\num{17.909202482974}$) at certain learning rates. Moreover, both RSOF and RVOS exhibit rapid convergence when a relatively large learning rate is used. In particular, RVOS reaches the prescribed accuracy with fewer iterations than RSOF on average over the learning rates considered in Table~\ref{tab3_1b}, suggesting that the VOS flow together with interpolation may accelerate convergence. Figure~\ref{fig3_1a} shows the convergence of function value of RSOF and RVOS with $\Delta t=0.1$. Both curves decrease in an oscillatory fashion, indicating that the algorithms effectively escape local minima while progressing toward the global minimizer.

To examine algorithm performance under random initializations, we conduct 500 independent trials from random starting points. The numbers of successful runs for the adpative algorithms are summarized in Table~\ref{tab11}. On average, ARSOF and ARVOS achieve markedly higher success rates than the competing methods. The distribution of the iterations for ARSOF, ARVOS, ADAM with $\Delta t=0.01$ are shown in Figure~\ref{fig3_1d}. ARVOS and ARSOF are more stable and require fewer average iteration steps than ADAM. Moreover, ARVOS has a slightly lower average number of iteration steps than ARSOF, which may be due to its use of the VOS formulation and interpolation. To enhance stochastic exploration and escape from poor local minima, we can introduce stochastic perturbations to the dynamical systems. The corresponding stochastic variant of the LLM-SOF scheme \eqref{eq2_5} reads as:
\begin{subequations}
\begin{empheq}[left=\empheqlbrace]{align}
&\frac{v^{k+1} - \frac{\xi^{k+1}}{\xi^{k}}v^k}{\Delta t} = - \xi^{k+1}\nabla f(x^k) - S(\mathcal{L}(x^{k+1} - x^k)) -\eta^{k+1}v^{k+1} + \dfrac{\sigma\epsilon_{k+1}}{\sqrt{\Delta t}}, \\
&\frac{x^{k+1} - x^k}{\Delta t} = v^{k+1},\\
&\frac{\xi^{k+1}\mathcal{E}^k - \xi^k\mathcal{E}^{k-1}}{\Delta t} = -\eta^{k+1}\left\|v^{k+1}\right\|^2,
\end{empheq}
\end{subequations}
where $\epsilon_{k+1} \sim \mathcal{N}(0, I)$ is a vector of i.i.d. standard Gaussian random variables. $\sigma$ is the noise level. The stochastic variants of the LLM-VOS algorithm can be constructed in a similar way. Table~\ref{tab_random} shows the numbers of successful runs under 500 random initializations when using the stochastic variants of the ARSOF and ARVOS to solve the problem. 


\begin{table}[!t]
	\renewcommand{\arraystretch}{1.1}
	\centering
	\caption{Loss value $|f(x)-f^*|$ using different methods with different initial learning rates for Rastrigin function}\label{tab3_1a}
	\begin{tabular}{c c c c c }
		\hline
		Method $\backslash$ $\Delta t$ & $0.0001$ & $0.001$ & $0.01$  &  $0.1$  \\
		\hline
		RSOF & $\num{0.00005662600205980080000 }$ & $\num{0.00008670197539473170000 }$ & $\num{0.00000453327524141400000 }$ & $\num{0.00003537583069146420000 }$  \\
		\hline
		RVOS & $\num{0.00006953722728297860000 }$ & $\num{0.00002320780612663490000 }$ & $\num{0.00006173300875644820000 }$ & $\num{0.00000410820651453036000 }$  \\
		\hline
		RSAV & $\num{17.909202482974}$ & $\num{17.909202482974}$ & $\num{17.909202482974}$ & $\num{49.74827272427}$  \\
		\hline
		RLLM & $\num{17.909202482974}$ & $\num{17.909202482974}$ & $\num{21.5032769767519}$ & $\num{44.3070985794344}$  \\
		\hline
		ADAM & $\num{17.909202482974}$ & $\num{17.909202482974}$ & $\num{17.909202482974}$ & $\num{17.9092024829766}$  \\
		\hline
		NAG & $\num{17.909202482974}$ & $\num{17.909202482974}$ & $\num{57.9599009526819}$ & $\num{9.25053131921331}$ \\
		\hline
		GD & $\num{17.909202482974}$ & $\num{17.909202482974}$ & $\num{50.1443707094346}$ & $\num{57.8494274515717}$ \\
		\hline
	\end{tabular}
\end{table}

\begin{table}[!t]
	\renewcommand{\arraystretch}{1.1}
	\centering
	\caption{The number of iterations reaching the specified accuracy using RSOF method (Algorithm \ref{alg:RSOF_scheme}) and the RVOS method (Algorithm \ref{alg:RVOS_scheme}) with different initial learning rates for Rastrigin function.}\label{tab3_1b}
	\begin{tabular}{c c c c c c }
		\hline
		Method $\backslash$ $\Delta t$ & $0.0001$ & $0.001$ & $0.01$  &  $0.1$ & Mean  \\
		\hline
		RSOF & $7351$ & $2160$ & $353$ & $88$ & $2488$  \\
		\hline
		RVOS & $817$ & $99$ & $25$ & $438$ & $344.75$  \\
		\hline		
	\end{tabular}
\end{table}

\begin{figure}[htbp]
	\centering
	\subfigure[RSOF]{
		\begin{minipage}[c]{0.4\linewidth}
			\centering
			\includegraphics[scale=0.32]{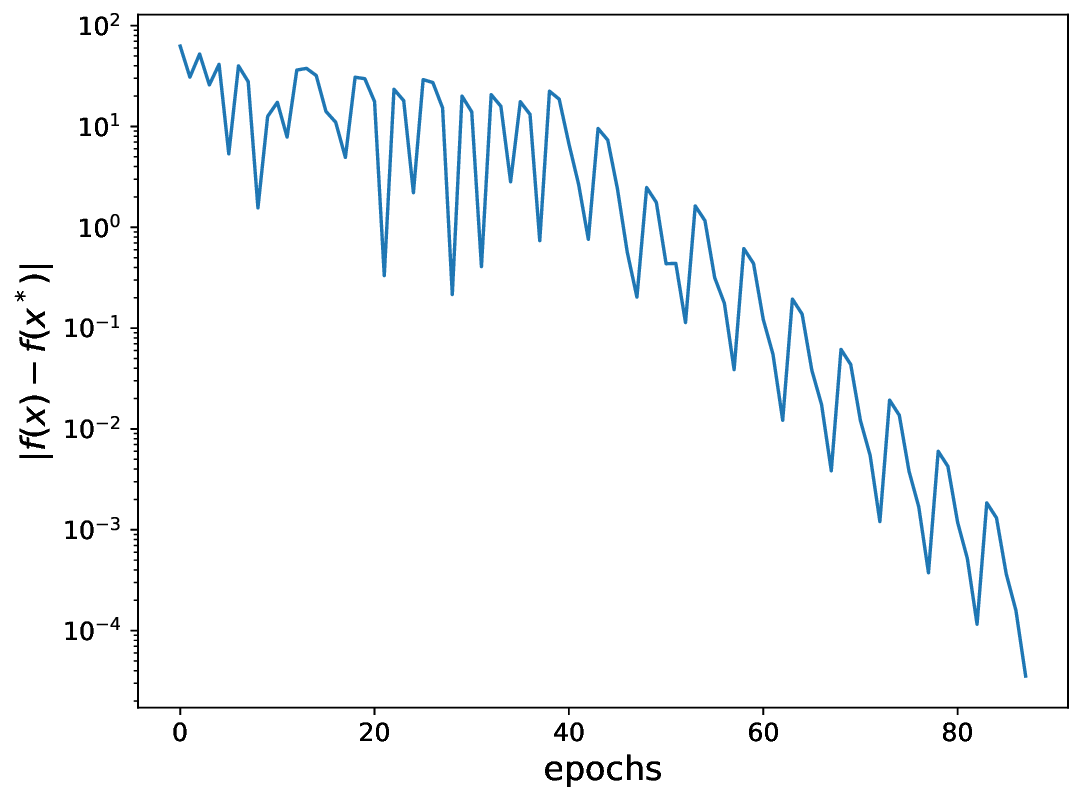}
	\end{minipage}}
	\subfigure[RVOS]{
		\begin{minipage}[c]{0.4\linewidth}
			\centering
			\includegraphics[scale=0.32]{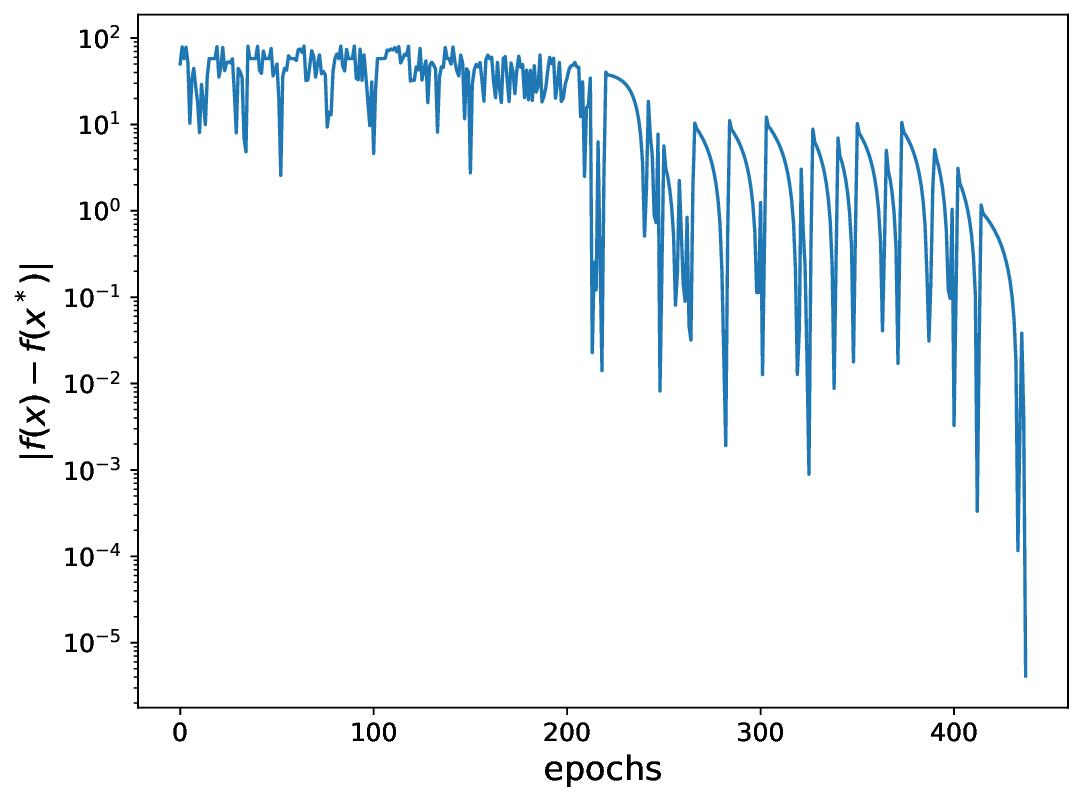}
	\end{minipage}}
	\caption{Convergence curves for Rastrigin function with $\Delta t=0.1$}\label{fig3_1a}
\end{figure}

\begin{table}[!t]
	\renewcommand{\arraystretch}{1.1}
	\centering
	\caption{Number of successful trials in 500 randomized tests for Rastrigin function}\label{tab11}
	\begin{tabular}{c c c c c c }
		\hline
		Method $\backslash$ $\Delta t$ & $0.0001$ & $0.001$ & $0.01$  &  $0.1$ & Mean \\
		\hline
		ARSOF & $43$ & $34$ & $38$ & $35$ &$37.5$   \\
		\hline
        ARVOS & $41$ & $37$ & $38$ & $38$  &$38.5$ \\
		\hline
		ARSAV & $41$ & $37$ & $1$ & $4$  &$20.75$\\
		\hline
		ARLLM & $41$ & $37$ & $0$ & $0$ &$19.5$\\
		\hline
		ADAM & $0$ & $2$ & $38$ & $14$ &$13.5$\\
		\hline
		NAG & $41$ & $37$ & $0$ & $0$ &$19.5$\\
		\hline
		GD & $31$ & $37$ & $0$ & $0$  &$17$\\
		\hline		
	\end{tabular}
\end{table}

\begin{table}[!t]
	\renewcommand{\arraystretch}{1.1}
	\centering
	\caption{Number of successful trials in 500 randomised tests for Rastrigin function with stochastic perturbations.}\label{tab_random}
	\begin{tabular}{c c c c c c }
		\hline
		Method $\backslash$ $\Delta t$ & $0.0001$ & $0.001$ & $0.01$  &  $0.1$ & Mean \\
		\hline
        $\sigma$ & $\num{10000}$ & $\num{1000000}$ & $\num{4000}$ & $\num{1000}$ &$-$   \\
		\hline
		ARSOF & $335$ & $242$ & $232$ & $312$ &$280.25$   \\
		\hline
        ARVOS & $447$ & $263$ & $311$ & $273$  &$323.5$ \\
		\hline		
	\end{tabular}
\end{table}



\begin{figure}[htbp]
	\centering
	\subfigure{
		\begin{minipage}[c]{0.4\linewidth}
			\centering
			\includegraphics[scale=0.32]{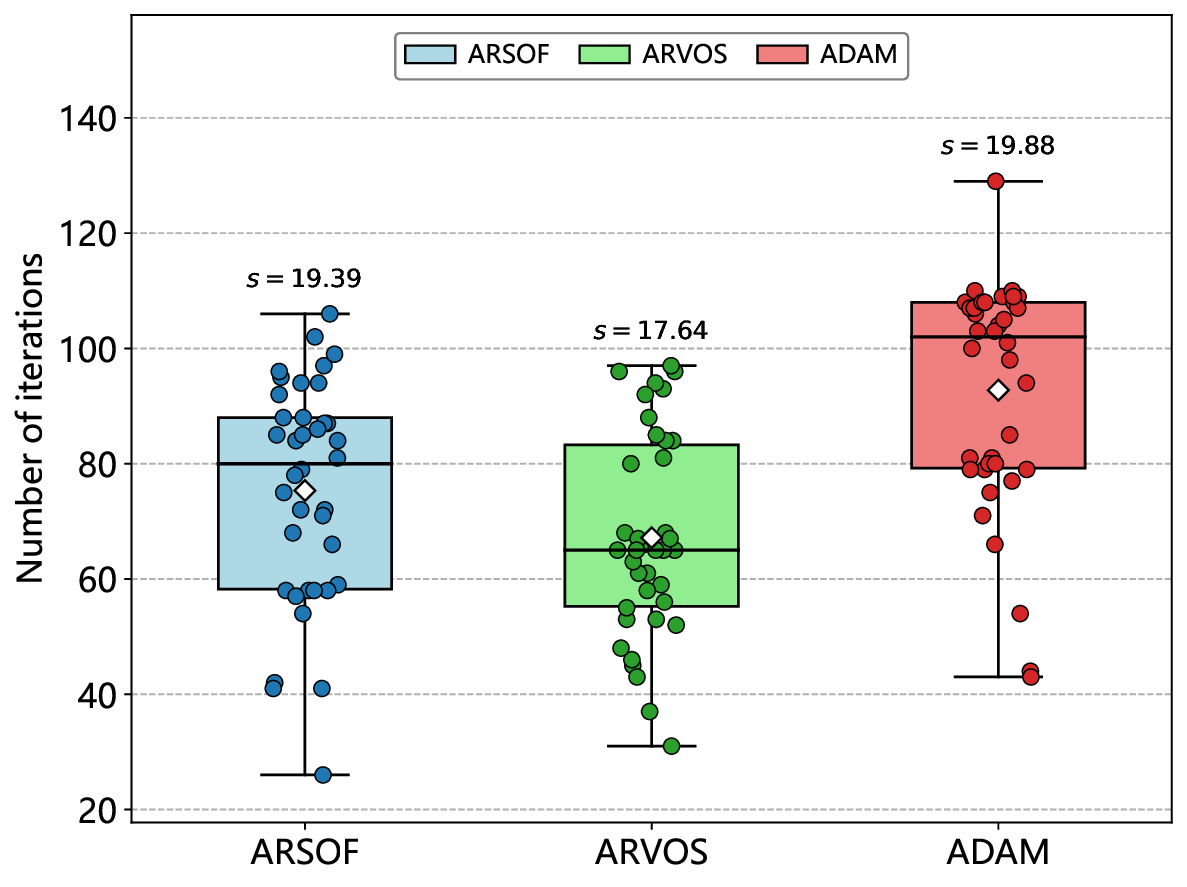}
	\end{minipage}}
	\subfigure{
		\begin{minipage}[c]{0.4\linewidth}
			\centering
			\includegraphics[scale=0.32]{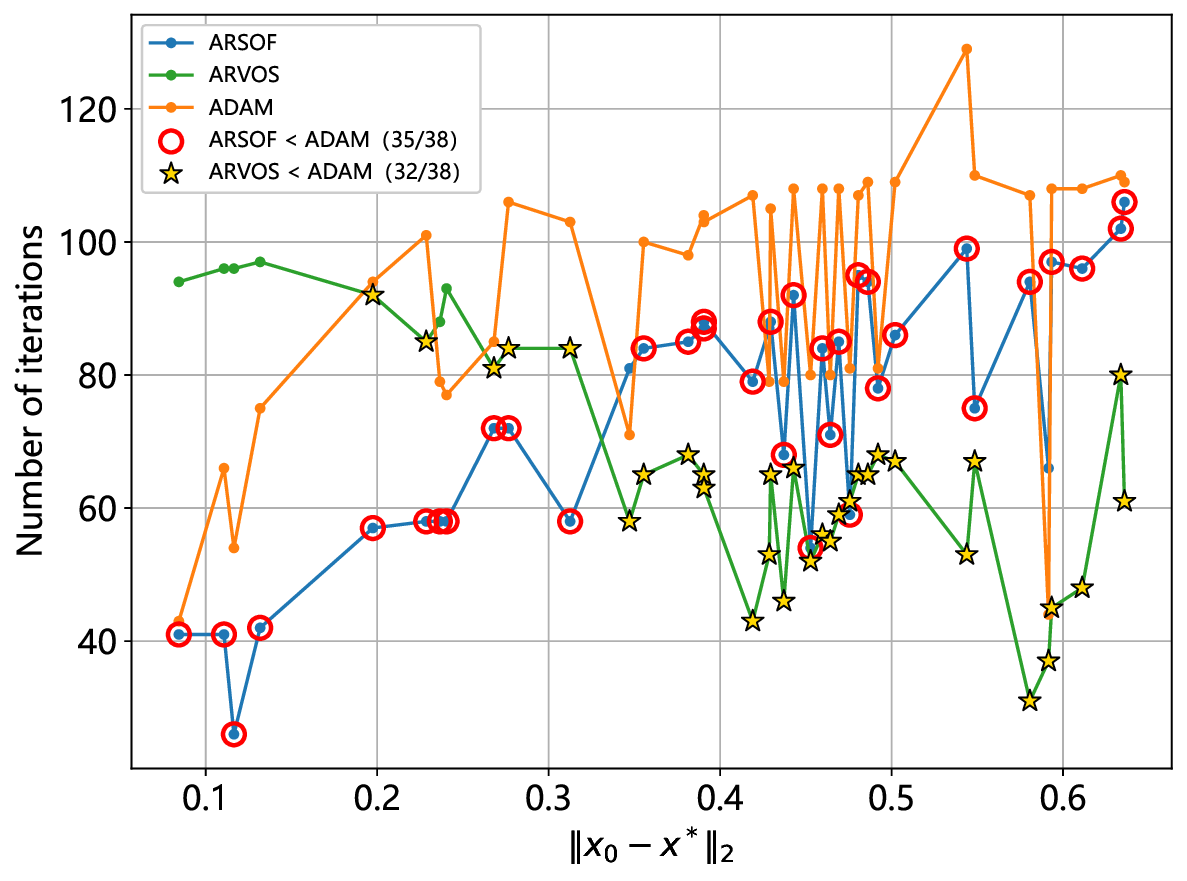}
	\end{minipage}}
	\caption{Distribution of the number of iterations for Rastrigin function with $\Delta t=0.01$}\label{fig3_1d}
\end{figure}

\subsection{Applications for solving PDEs}\label{section3_2}
In this section, we present a series of numerical experiments on the numerical solutions of PDEs, including the Poisson equation, and Burgers' equation. In Section~\ref{subsub2a}, the proposed adaptive algorithms are implemented within PINNs framework for linear PDEs problems. In Section~\ref{subsub2c}, we further demonstrate the versatility of the proposed algorithms by integrating them with DeepONets framework to solve the nonlinear Burgers' equation.

All numerical experiments are implemented in the PyTorch deep learning framework~\cite{paszke2019pytorch}. The hyperbolic tangent function (\texttt{tanh}) is used as the activation function, and Kaiming uniform initialization is adopted for networks parameter initialization unless otherwise specified. To ensure fair comparison of performance across different methods, a unified fully connected feedforward NN (FNN) architecture is employed for all tests in Sections~\ref{subsub2a}, consisting of four hidden layers with 128 neurons in each layer. A total of 10,000 training samples are generated via uniform random sampling, and the batch size is set to 200. For the experiments in Section~\ref{subsub2c}, we use the same network architecture as that in~\cite{lu2021learning}, with a training set of size 500, a validation set of size 100, and a batch size of 500. Unless otherwise stated, we set $S=1$ and $\mathcal{L}=\mathcal{I}$, where $\mathcal{I}$ denotes the identity matrix. To account for stochasticity in neural network training, every numerical experiment in Section~\ref{section3_2} is repeated for five independent runs.

We evaluate the accuracy by using the pointwise relative $L^2$-error, defined as follows:
\begin{align*}
	\|e\|_{L^2} = \frac{\sqrt{\sum_{i=1}^N |u_\theta(X_i) - u^*(X_i)|^2}}{\sqrt{\sum_{i=1}^N |u^*(X_i)|^2}},
\end{align*}
where $X_i$ is a sampling test point, $u^*$ is the exact solution, and $u_\theta$ is the approximate solution. If the training process fails due to an inappropriate choice of the (initial) learning rate, we indicate this with NAN.

\subsubsection{Poisson equation}\label{subsub2a}
In this test, we consider the $d$-dimensional Poisson equation on the domain $\Omega=[-0.5,0.5]^d$ with homogeneous Dirichlet
boundary conditions:
\begin{equation*}
	\Delta u = -d\pi^2\prod_{i=1}^d \cos(\pi x_i) , \quad \mathbf{x} \in \Omega,
\end{equation*}
for which the exact solution reads:
\begin{equation*}
	u(\mathbf{x}, t) =\prod_{i=1}^{d} \cos(\pi x_i).
\end{equation*}
In this example, we set $d=2$.

Table~\ref{tab3_3a} summarizes the mean $L^2$ errors $\|e\|_{L^2}$ of PINNs (ARSOF), PINNs (ARVOS), PINNs (ADAM), PINNs (NAG), and PINNs (SGD) for the Poisson equation across a range of initial learning rates. PINNs (ARSOF) and PINNs (ARVOS) remain insensitive to the choice of initial learning rate $\Delta t$, and consistently deliver higher accuracy than the competing methods, by at least one order of magnitude. 


Figures~\ref{fig3_3a} presents the $L^2$ error curves and and~\ref{fig3_3b} plots the error distribution of five independent runs for each method under different values of $\Delta t$. As can be observed, PINNs (ARSOF) and PINNs (ARVOS) exhibit markedly more stable convergence behavior than the other methods. In particular, PINNs (ARSOF) achieves the highest accuracy among all compared methods. 


Finally, we examine the point-wise error distributions with $\Delta t=0.1$, as shown in Figures~\ref{fig3_3f}. The results demonstrate the robustness and stability of PINNs (ARSOF) and PINNs (ARVOS) with large learning rates.

\begin{table}[!t]
	\renewcommand{\arraystretch}{1.1}
	\centering
	\caption{Mean value of $\left\|e\right\|_{L^2}$ using different methods with different initial learning
		rates for Poisson equation}\label{tab3_3a}
	\begin{tabular}{c c c c c c }
	\hline
	Method $\backslash$ $\Delta t$ & $0.0001$ & $0.001$ & $0.01$  &  $0.1$ & $1$ \\
	\hline
	PINNs (ARSOF) & $\num{0.00002287}$ & $\num{0.00004571}$ & $\num{0.00002121}$ & $\num{0.00002051}$ &$\num{0.00004329}$  \\
	\hline
    PINNs (ARVOS) & $\num{0.00006039}$ & $\num{0.00005057}$ & $\num{0.0000708}$ & $\num{0.0000603}$ & $\num{0.00070062}$  \\
	\hline
	PINNs (ADAM) & $\num{0.000504}$ & $\num{0.000128}$ & $\num{0.122706}$ & $\num{0.122996}$ & $\num{0.316025}$ \\
	\hline
	PINNs (NAG) & $\num{0.000198}$ & $\num{0.000029}$ & $\num{0.000798}$ & \texttt{NaN} & \texttt{NaN} \\
	\hline
	PINNs (SGD) & $\num{0.00049}$ & $\num{0.000202}$ & $\num{0.000545}$ & \texttt{NaN} & \texttt{NaN} \\
	\hline
\end{tabular}
\end{table}

\begin{figure}[htbp]
	\centering
	\subfigure[$\Delta t=0.0001$]{
		\begin{minipage}[c]{0.4\linewidth}
			\centering
			\includegraphics[scale=0.32]{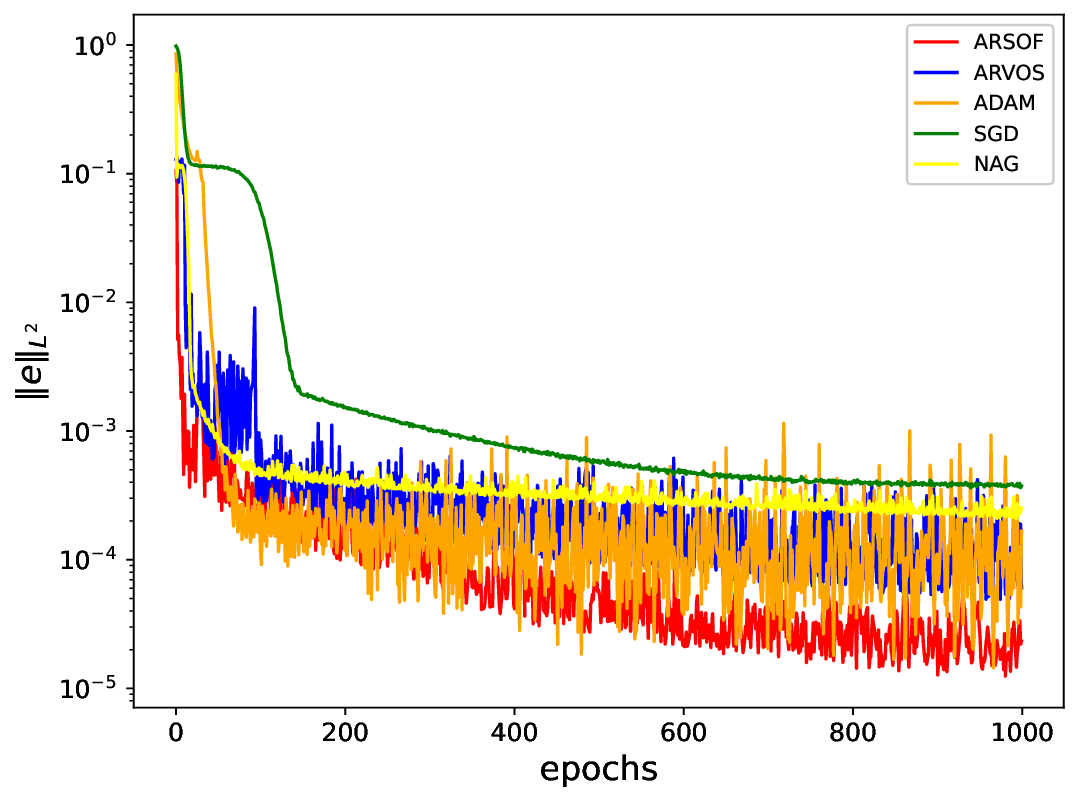}
	\end{minipage}}
	\subfigure[$\Delta t=0.001$]{
		\begin{minipage}[c]{0.4\linewidth}
			\centering
			\includegraphics[scale=0.32]{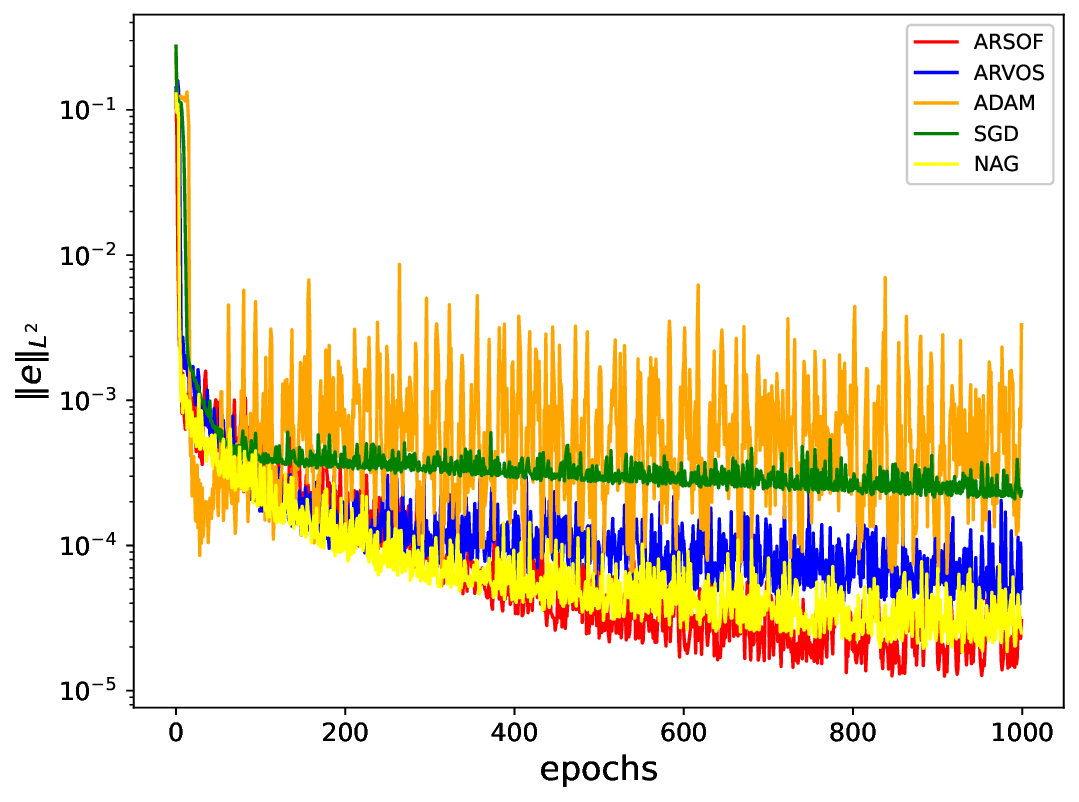}
	\end{minipage}}
	\caption{The error curves of different methods with different $\Delta t$ for Poisson equation}\label{fig3_3a}
\end{figure}

\begin{figure}[htbp]
	\centering
	\subfigure[$\Delta t=0.0001$]{
		\begin{minipage}[c]{0.4\linewidth}
			\centering
			\includegraphics[scale=0.32]{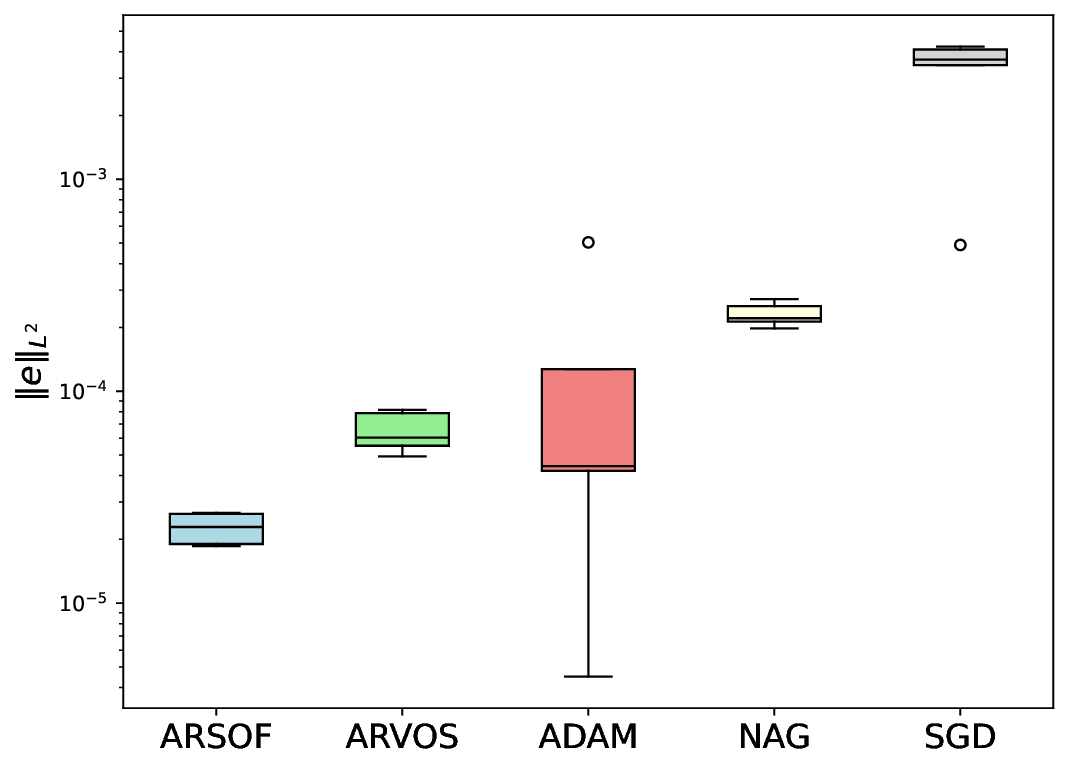}
	\end{minipage}}
	\subfigure[$\Delta t=0.001$]{
		\begin{minipage}[c]{0.4\linewidth}
			\centering
			\includegraphics[scale=0.32]{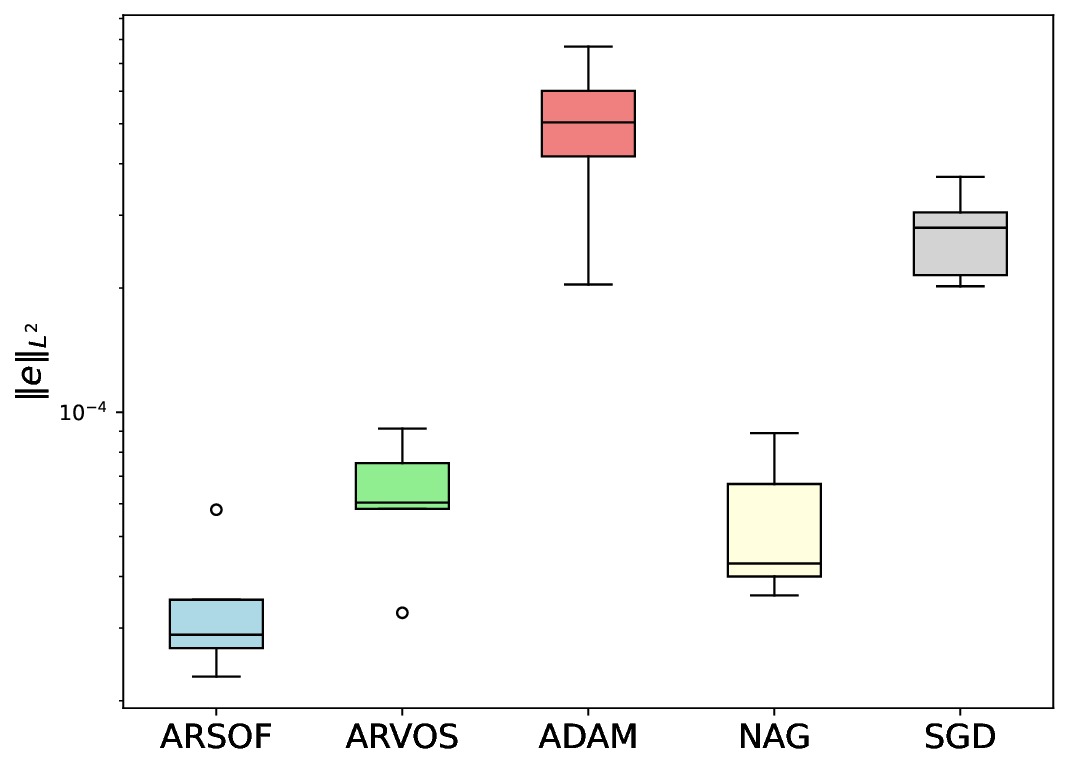}
	\end{minipage}}
	\caption{The error box plots of different methods with 5 independent runs and different $\Delta t$ for Poisson equation}\label{fig3_3b}
\end{figure}

\begin{figure}[htbp]
	\centering
	\subfigure[PINNs (ARSOF)]{
		\begin{minipage}[c]{0.4\linewidth}
			\centering
			\includegraphics[scale=0.32]{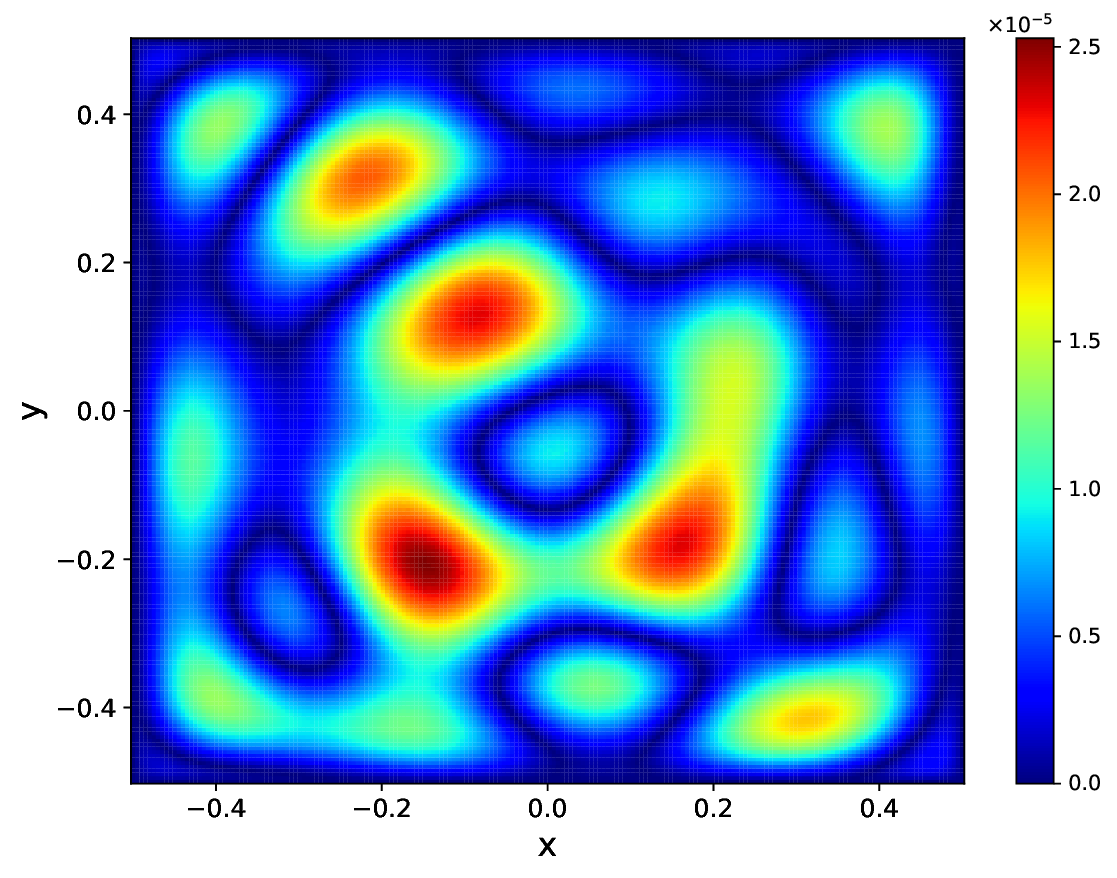}
	\end{minipage}}
    \subfigure[PINNs (ARVOS)]{
		\begin{minipage}[c]{0.4\linewidth}
			\centering
			\includegraphics[scale=0.32]{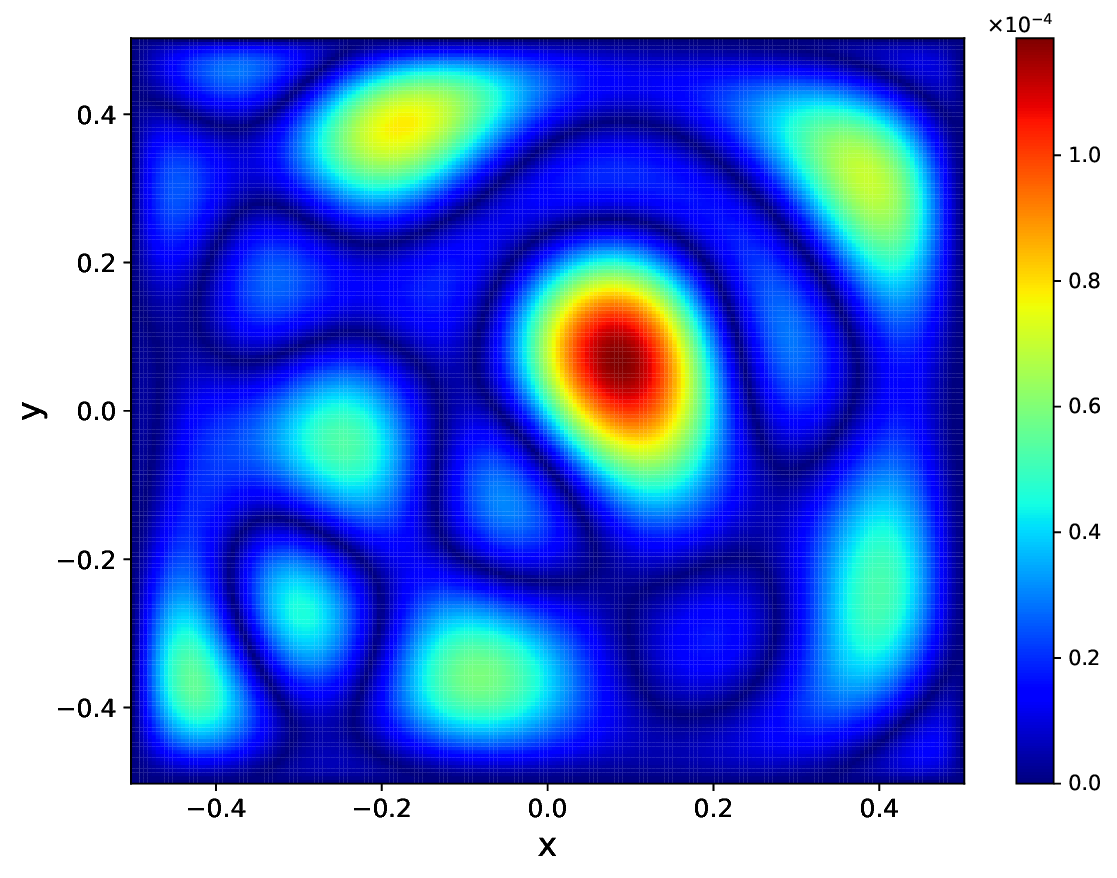}
	\end{minipage}}
	\caption{Point-wise error of PINNs (ARSOF) and PINNs (ARVOS) with $\Delta t=0.1$ for Poisson equation}\label{fig3_3f}
\end{figure}

\subsubsection{Burgers' equation}\label{subsub2c}
Consider the following one-dimensional Burgers' equation:
\begin{equation*}
	\frac{\partial u}{\partial t} + u \frac{\partial u}{\partial x} = \nu \frac{\partial^2 u}{\partial x^2},\quad x\in(0,1),\ t\in(0,1],
\end{equation*}
with periodic boundary condition, where the viscosity $\nu = 0.001$. Here, we learn the operator mapping from the initial condition $u(x,0)=u_0(x)$ to the solution $u(x,t)$ at $t=1$, i.e.,
\begin{equation*}
\mathcal{G}: u_0(x) \mapsto u(x,1).
\end{equation*}

We use the dataset generated in \cite{li2021fourier}, where the initial condition is sampled from a Gaussian random field with a Riesz kernel, denoted by $\mu = \mathcal{R}\big(0, 625(-\Delta + 25I)^{-2}\big)$. Here, $\mu$ is the probability measure on the function space, and $\Delta$ and $I$ represent the Laplacian and the identity, respectively. And we use a spatial resolution with 128 grids to represent the input and output functions. Table~\ref{tab3_5a} demonstrates the mean $L^2$ errors $\|e\|_{L^2}$ for DeepONets (ARSOF), DeepONets (ARVOS), DeepONets (ADAM), DeepONets (NAG), and DeepONets (SGD) across all tested initial learning rates. DeepONets (ARSOF) and DeepONets (ARVOS) consistently display more stable convergence behavior than the competing algorithms. 

While DeepONets (ADAM) yields the lowest training loss as presented in Figure~\ref{fig3_5b}(a), its corresponding validation loss in Figure~\ref{fig3_5b}(b) remains higher. Such a pronounced discrepancy between training and validation performance is a typical signature of overfitting. By contrast, neither DeepONets (ARSOF) nor DeepONets (ARVOS) exhibits overfitting behavior, and both achieve consistently higher solution accuracy than the baseline methods. Figure~\ref{fig3_5c} compares the numerical solutions produced by the five methods with $\Delta t=0.001$ against the reference solution. The solutions generated by DeepONets (ARSOF) and DeepONets (ARVOS) align closely with the ground-truth profile, whereas the results from the other three methods exhibit visible deviations. 


\begin{table}[!t]
	\renewcommand{\arraystretch}{1.1}
	\centering
	\caption{Mean value of $\left\|e\right\|_{L^2}$ using different methods with different initial learning
		rates for Burgers' equation}\label{tab3_5a}
	\begin{tabular}{c c c c c c }
		\hline
		Method $\backslash$ $\Delta t$ & $0.0001$ & $0.001$ & $0.01$  &  $0.1$ & $1$ \\
		\hline
		DeepONets (ARSOF) & $\num{0.0542}$ & $\num{0.0553}$ & $\num{0.0551}$ & $\num{0.0538}$ & $\num{0.0588}$  \\
		\hline
        DeepONets (ARVOS) & $\num{0.0663}$ & $\num{0.0890}$ & $\num{0.0971}$ & $\num{0.0920}$ & $\num{0.1114}$  \\
		\hline
		DeepONets (ADAM) & $\num{0.0525}$ & $\num{0.1871}$ & $\num{2.8584}$ & $\num{37.8654}$ & $\num{1371.8444}$ \\
		\hline
		DeepONets (NAG) & $\num{0.2313}$ & $\num{0.1123}$ & $\num{0.0661}$ & \texttt{NaN} & \texttt{NaN} \\
		\hline
		DeepONets (SGD) & $\num{0.3130}$ & $\num{0.2313}$ & $\num{0.1116}$ & $\num{0.0680}$ & \texttt{NaN} \\
		\hline
	\end{tabular}
\end{table}


\begin{figure}[htbp]
	\centering
	\subfigure[Training loss]{
		\begin{minipage}[c]{0.38\linewidth}
			\centering
			\includegraphics[scale=0.3]{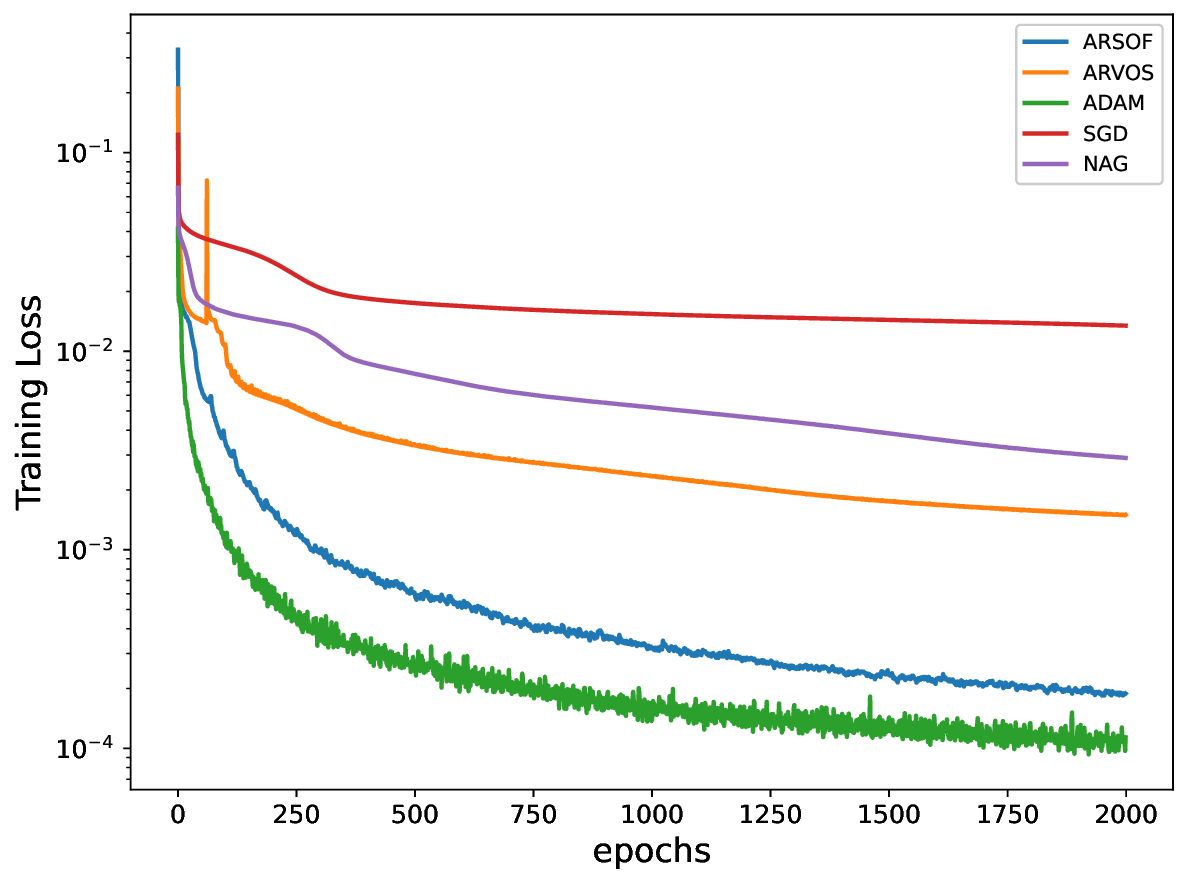}
	\end{minipage}}
	\subfigure[Validation loss]{
		\begin{minipage}[c]{0.38\linewidth}
			\centering
			\includegraphics[scale=0.3]{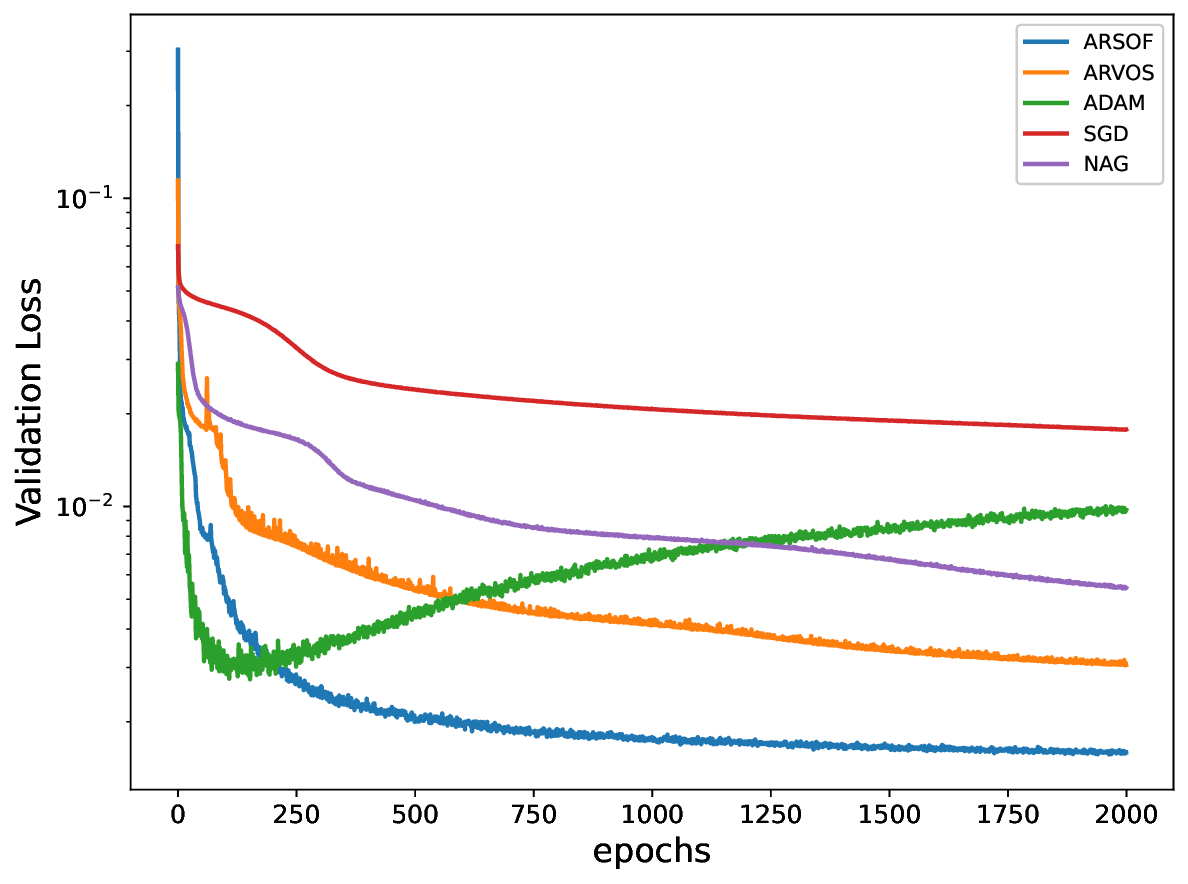}
	\end{minipage}}
	\caption{The loss curves of five methods with $\Delta t=0.001$ for Burgers' equation}\label{fig3_5b}
\end{figure}

\begin{figure}[htbp]
	\centering
	\subfigure[Initial condition]{
		\begin{minipage}[c]{0.38\linewidth}
			\centering
			\includegraphics[scale=0.3]{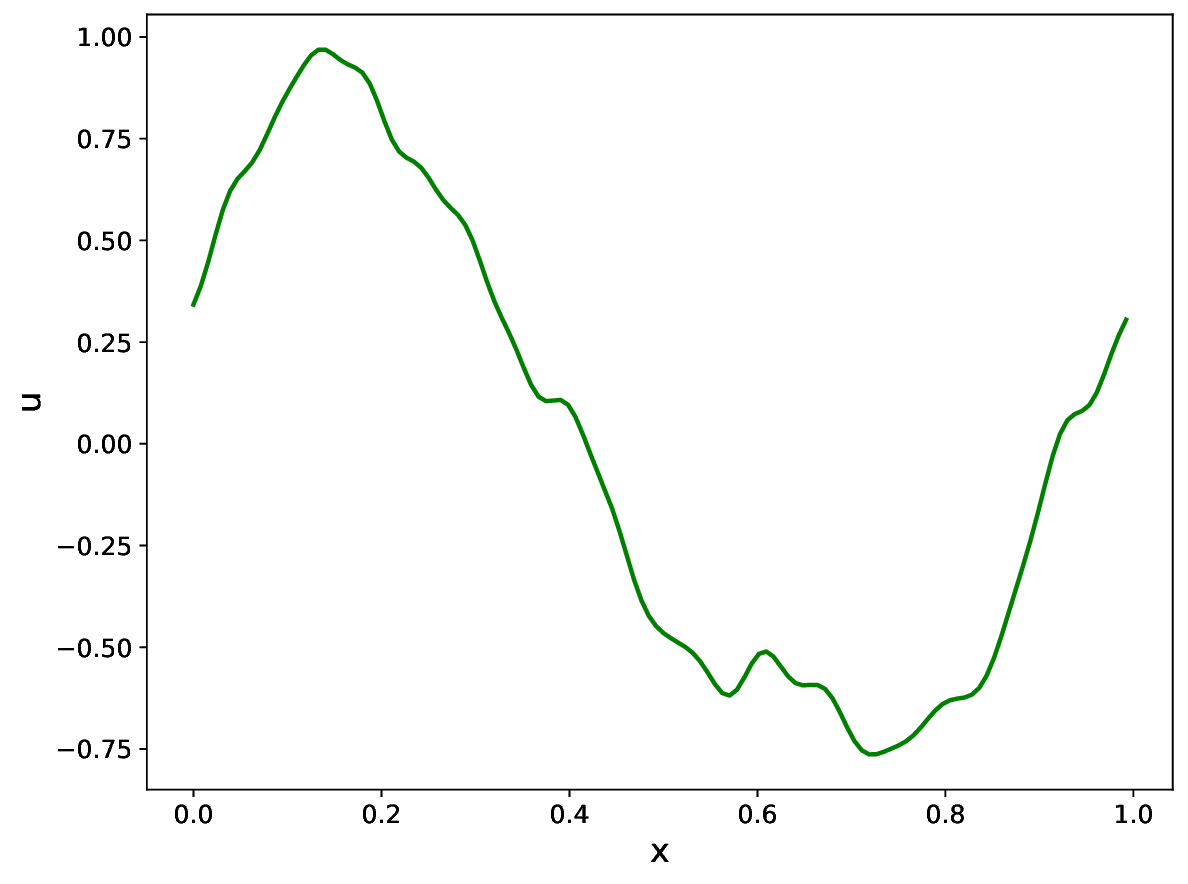}
	\end{minipage}}
	\subfigure[Burgers solution at final time]{
		\begin{minipage}[c]{0.38\linewidth}
			\centering
			\includegraphics[scale=0.3]{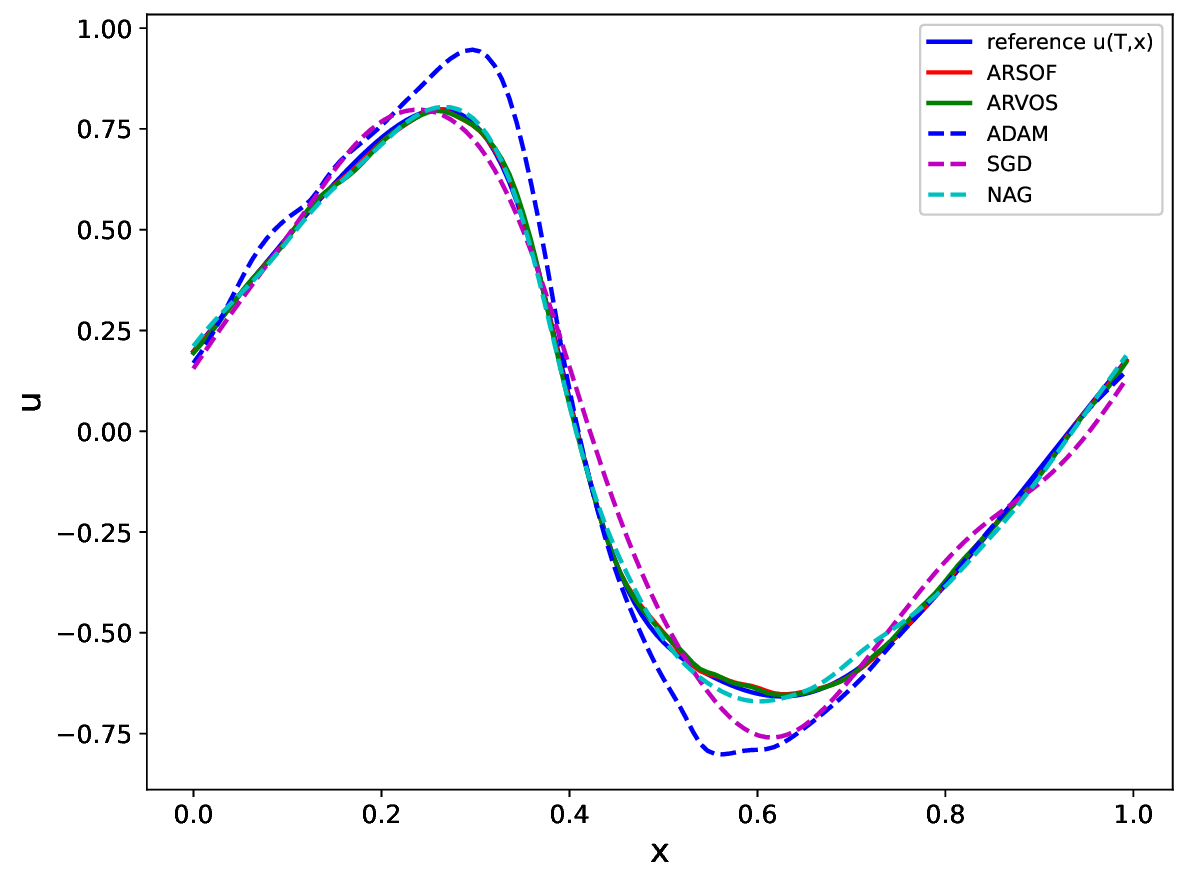}
	\end{minipage}}
	\caption{The numerical solutions of five methods with $\Delta t=0.001$ for Burgers' equation}\label{fig3_5c}
\end{figure}

\section{Concluding remarks}
In this paper, we propose two robust and efficient optimization algorithms based on the second-order flow by using a novel Lagrange multiplier framework and variable and operator splitting approach. The main idea of the proposed methods is to reformulate unconstrained minimization problems as the search for steady-state solutions of the second-order flow. From this perspective, we construct two unconditionally modified pseudo-energy-stable Lagrange multiplier schemes and further develop their relaxed and adaptive variants.
We first apply the proposed methods to the Rosenbrock and Rastrigin functions in order to demonstrate the capability of the relaxed and adaptive algorithms, namely RSOF, RVOS, ARSOF and ARVOS, in finding the global minima of non-convex functions. The proposed approaches are incorporated into the PINNs and DeepONets frameworks as optimizers for solving linear and nonlinear PDEs. Numerical experiments demonstrate that the RSOF, RVOS, ARSOF, and ARVOS algorithms are highly robust with respect to the choice of the initial learning rate, thereby significantly improving the efficiency, robustness, and accuracy of the training process. Moreover, the proposed methods maintain great high accuracy even under large initial learning rates and consistently outperform existing optimization methods, including ADAM, SGD, and NAG.

In the present work, our primary focus is on applying the proposed optimization strategies within the PINNs and DeepONets frameworks. Although satisfactory numerical results have been obtained, the solution accuracy can still be further improved. It is well known that numerous effective techniques have been developed to enhance the accuracy, including gradient-enhanced strategies, adaptive weighting methods, and adaptive sampling techniques. Furthermore, the methods we have proposed still rely solely on first-order information about the loss function, and we have not conducted an in-depth study of the choice of the operator $\mathcal{L}$, which may further limit the accuracy and convergence of our method for solving PDEs. Finally, we have not provided a sufficient theoretical analysis of why VOS-based methods sometimes have an accelerating effect. In future work, we plan to combine the proposed optimization algorithms with some advanced ideas to further improve solution accuracy, extend the applicability of our methods to more challenging PDEs and present further theoretical results.

\section{Declaration of interests}
The authors report no conflict of interest.
\bibliographystyle{siamplain}
\bibliography{Sec-order-flow}

\end{document}